\documentclass{article}
\usepackage{amsmath,amsfonts,xcolor,tikz,amsthm,xifthen,hyperref,graphicx}
\hypersetup{linkbordercolor={1 .8 .8},citebordercolor={.7 1 .7},runbordercolor={.8 .8 1}}
\usetikzlibrary{shapes,positioning}
\newcommand{\vart}{\mathcal T}
\newcommand{\vars}{\mathcal S}
\newcommand{\sh}{\text{sh}}
\newcommand{\ve}{\mathbf{e}}
\newcommand{\ev}{\text{ev}}
\newcommand{\real}{\mathbb R}
\newcommand{\col}{{\text{col}}}
\newcommand{\row}{{\text{row}}}
\newcommand{\AND}{{\text{ AND }}}
\newcommand{\NOT}{{\text{ NOT }}}
\newtheorem{thm}{Theorem}
\newtheorem{prop}[thm]{Proposition}
\begin{document}
\author{Yuchen Pei}
\date{Mathematics Institute, \\University of Warwick, \\Coventry CV4 7AL, UK\\y.pei@warwick.ac.uk}
\title{A symmetry property for $q$-weighted Robinson-Schensted and other branching insertion algorithms}
\maketitle
\begin{abstract} 
In \cite{op13} a $q$-weighted version of the Robinson-Schensted algorithm was introduced. In this paper we show that this algorithm has a symmetry property analogous to the well known symmetry property of the normal Robinson-Schensted algorithm. The proof uses a generalisation of the growth diagram approach introduced by Fomin \cite{fomin79,fomin86,fomin94,fomin95}. This approach, which uses ``growth graphs'', can also be applied to a wider class of insertion algorithms which have a branching structure, including some of the other $q$-weighted versions of the Robinson-Schensted algorithm which have recently been introduced by Borodin-Petrov \cite{bp13}.
\end{abstract}
\section{Introduction}
In \cite{op13} a $q$-weighted version of the Robinson-Schensted algorithm was introduced.  
In this paper we show that this algorithm enjoys a symmetry property analoguous to the well-known symmetry property of the Robinson-Schensted algorithm.
The proof uses a generalisation of the growth diagram approach introduced by \cite{fomin79,fomin86,fomin94,fomin95}. 

The insertion algorithm we consider in this paper is based on column insertion, but the technique applies to any insertion algorithm belonging to a certain class of ``branching insertion algorithms'', as described in Section \ref{s:other} below.  
For example, \cite{bp13} have recently introduced a $q$-weighted version of the row insertion algorithm, which is defined similarly to the column insertion version of \cite{op13};  they also consider a wider family of such algorithms (and, more generally, dynamics on Gelfand-Tsetlin patterns), some of which fall into the framework considered in the present paper, and can 
similarly be shown to have the symmetry property.  
We discuss such extensions in Section \ref{s:other} below.

The Robinson-Schensted (RS) algorithm is a combinatorial algorithm which was introduced by Robinson \cite{robinson} and Schensted \cite{schensted}. 
It has wide applications in representation theory and probability theory, e.g. last passage percolation, totally asymmetric simple exclusion process / corner growth model, random matrix theory \cite{johansson}, queues in tandem \cite{noc03} and more. 

There are two versions of RS algorithms, the row insertion and column insertion version. In most of this paper we deal with column insertion and its $q$-version.
The RS algorithm transforms a word to a tableau pair of the same shape.
A word can be treated as a path, hence a random word corresponds to a random walk.
When taking such a random walk, the shape of the output tableaux is a Markov chain, whose transition kernel is related to the Schur symmetric functions \cite{noc1}.
A geometric generalisation transforms a Brownian motion with drift to a Markov process whose generator is related to $\mathsf{GL}(n,\mathbb R)$-Whittaker functions \cite{noc12}, which are eigenfunctions of the quantum Toda chain \cite{kostant77}.
The $q$-Whittaker functions on the one hand are a generalisation of the Schur symmetric functions and a specialisation of Macdonald $(q,t)$-symmetric functions when $t=0$ \cite{mac98}, and on the other hand are eigenfunctions of the $q$-deformed quantum Toda chain \cite{ruijsenaars90,etingof99}.
When $q\to0$ they become the Schur functions and when $q\to1$ with a proper scaling \cite{glo12} they converge to Whittaker functions.
In the spirit of this connection, a $q$-weighted Robinson-Schensted algorithm was formulated in \cite{op13}, which transforms the random walk to a Markov chain that is related to $q$-Whittaker functions.

As is expected the algorithm degenerates to the normal RS algorithm when $q\to0$. Part of the random $P$-tableau also has $q$-TASEP dynamics \cite{op13}, the latter introduced in \cite{sw98}, just as the same part of the random $P$-tableau of normal RS algorithm has TASEP dynamics \cite{noc03}.

Recently, a $q$-version of the RS row insertion algorithm was also introduced in \cite{bp13}. It denegerates to the normal RS algorithm with row insertion when $q\to0$.

In this paper we show that both algorithms enjoy a symmetry property analogous to the well-known symmetry property of the Robinson-Schensted algorithm.
Basically, the symmetry property for the Robinson-Schensted algorithms restricted to permutation inputs is the property that the output tableau pair is interchanged if the permutation is inversed.
Knuth \cite{knuth} generalised the normal row insertion algorithm to one which takes matrix input, which we refer to as Robinson-Schensted-Knuth (RSK) algorithm. For this algorithm the symmetry property is that the output tableau pair is interchanged if the matrix is transposed.
Note that the matrix becomes the permutation matrix when the RSK algorithm is restricted to permutation, hence the transposition of the matrix corresponds to inversion of the permutation.
Burge gives a similar generalisation of the column insertion algorithm \cite{burge74}, which we refer to as Burge's algorithm.

The symmetry property for the normal RS algorithm is normally discussed in the literature for the row insertion. However, in the permutation case the output tableau pair for row insertion is simply the transposition of the pair for column insertion. Therefore proofs of the symmetry property can be translated to column insertion instantly.
In the matrix input case, the Burge and RSK algorithms are also closely related so that proofs can be extended from one to the other naturally.

For a proof of the symmetry property one can see e.g. \cite{sagan,stanley,fulton}. There are a few different approaches to deal with it. 
One is the Viennot diagram \cite{viennot77} which provides a nice geometric construction of the RS algorithms. 
In the matrix input case there are two approaches which reduce to the Viennot diagram approach when restricted to permutations. One is the antichain or the inversion digraph construction of the RSK algorithm (which should extend naturally to the Burge algorithm) due to \cite{knuth}; the other (for both RSK and Burge) is Fomin's matrix-ball construction of \cite{fulton}. 

Another method is the growth diagram technique due to Fomin \cite{fomin79,fomin86,fomin94,fomin95}. It can be generalised to the RSK and Burge algorithms.
Greene's theorem \cite{greene74} gives a proof by showing the relation between the lengths of the longest subsequences of the input and the shape of the output. However, the growth diagram approach can be thought of as a fast construction to calculate these lengths.

Among all these techniques, the special structure of growth diagram can be extended to a class of algorithms what we will call branching algorithms, which include the $q$-weighted column and row insertion algorithms.
Therefore it is this approach which we use in this paper. For a simple description of the technique for normal row insertion see e.g. \cite{stanley}, whose column version will be shown in Section \ref{s:symnormal}.

The rest of the paper is organised in the following way.
In Section \ref{s:classical} we recall the insertion rule of a letter into a tableau for the normal Robinson-Schensted algorithm (with column insertion). We describe it in a way that suits the growth diagram. 
In Section \ref{s:symnormal} we describe the insertion rule for a word and state the symmetry property for the RS algorithm applied to permutations.
In Section \ref{s:qrs} and \ref{s:wordqrs} we describe the $q$-weighted insertion algorithm for letters and words in a way which is different, but equivalent to the definition given in \cite{op13}. 
In Section \ref{s:symq} we state and prove the symmetry property in the $q$-case.
Finally in Section \ref{s:other} we prove the symmetry property for the row insertion algorithm and more generally branching algorithms.

\textit{Acknowledgements}. The author would like to thank Neil O'Connell for guidance. Research of the author is supported by EPSRC grant number EP/H023364/1.
\section{Classical Robinson-Schensted algorithm}\label{s:classical}
A partition $\lambda=(\lambda_1,\lambda_2,\cdots,\lambda_k)\in W=\{(a_1,a_2,\dots)\in \bigcup_{k=1}^\infty \mathbb N_{\ge0}^k,a_1\ge a_2\ge \dots\}$ is a vector of weakly decreasing non-negative integer entries. Denote by $l(\lambda)$ the number of positive entries of $\lambda$ and $|\lambda|=\lambda_1+\cdots+\lambda_{l(\lambda)}$ the size of the partition. We say $\lambda$ is a partition of $n$ if $|\lambda|=n$, which we denote by $\lambda\vdash n$. A Young tableau $P$ with shape $\lambda$ is a left aligned array of $l(\lambda)$ rows of positive integers such that the entries are strictly increasing along each column and weakly increasing along each row, and such that the length of the $j$th row is $\lambda_j$. For example below is a tableau with shape $(4,3,2,2)$.
\begin{align}\label{tab:ex1}
\begin{array}{cccc}
  1&1&3&4\\3&5&8&\\6&7&&\\8&8&&
\end{array}
\end{align}
We denote by $\sh P$ the shape of tableau $P$, and $\vart_\ell$ the set of all tableaux with entries no greater than $\ell$. Denote $[n]=\{1,2,\cdots,n\}$ for a positive integer $n$. A standard tableau $Q$ is a tableau with distinct entries from $[|\sh Q|]$. For example below is a standard tableau with shape $(4,3,2,2)$.
\begin{align*}
  \begin{array}{cccc}
    1&3&4&7\\2&5&8&\\6&10&&\\9&11&&
  \end{array}
\end{align*}
We denote by $\vars_n$ the set of standard tableaux with shape of size $n$. For example the above tableau is an element of $\vars_{11}$.

For a tableau $P$, we denote by $P^k$ its subtableau containing all entries no greater than $k$ and call it the $k$th subtableau of $P$. For example the $6$th subtableau of the tableau shown in \eqref{tab:ex1} is
\begin{align*}
  \begin{array}{cccc}
    1&1&3&4\\3&5&&\\6&&&
  \end{array}
\end{align*}
A tableau $P$ can be identified by the shape of its subtableaux, which we usually denote by $\lambda^k=\sh P^k$. Also let $\lambda^0=\emptyset$ to be the empty partition.
Evidently a tableau has only finitely many different $\lambda^i$'s and there exists an $\ell$ such that $\lambda^i=\sh P$ for $i\ge\ell$. We call $\lambda^i$ the $i$th shape of $P$.
We can identify $P$ with these shapes and write $P=\lambda^0\prec\lambda^1\prec\lambda^2\prec\dots\prec\lambda^\ell$, where ``$\prec$'' is an interlacing relation: for $a=(a_1,a_2,\dots)$ and $b=(b_1,b_2,\dots)$, $a\prec b$ means $b_1\ge a_1\ge b_2 \ge a_2 \ge\dots$.
For example the tableau $P$ in \eqref{tab:ex1} is identified as
\begin{align*}
  P=\emptyset\prec2\prec2\prec31\prec41\prec42\prec421\prec422\prec4322.
\end{align*}

The basic operation of the Robinson-Schensted algorithm is to insert a letter $k\in\mathbb N_+:=\mathbb N\backslash\{0\}$ into a tableau $P=\emptyset\prec\lambda^1\prec\lambda^2\prec\dots\dots\prec\lambda^\ell$ and produce a new tableau $\tilde P=\emptyset\prec\tilde\lambda^1\prec\tilde\lambda^2\prec\dots\prec\tilde\lambda^\ell$.
To do this we first find the lowest row in $\lambda^k$ such that appending a box at the end of that row would preserve the interlacement between the $k-1$th shape and the $k$th shape, and append the box to that row.
Suppose the row has index $j_k$, then we find the lowest row in $\lambda^{k+1}$ that is no lower than row $j_k$ such that appending a box at the end of that row would preserve the interlacement between the $k$th shape and the $k+1$th shape, and append the box to that row, and so on and so forth. 
More precisely, define
\begin{align*}
  j_{k-1}=k;\quad
  j_{i}=\max(\{j\le j_{i-1}:\lambda^{i-1}_{j-1}>\lambda^i_j\}\cup\{1\}),\qquad i\ge k.
\end{align*}
Then the new tableau $\tilde P$ is defined by
\begin{align*}
  \tilde\lambda^i=
  \begin{cases}
    \lambda^i,&\text{if }i<k;\\
    \lambda^i+\ve_{j_i},&\text{otherwise.}
  \end{cases}
\end{align*}
where $\ve_{j}$ is the $j$th standard basis of $\real^{\mathbb N_+}$.

For example, if we insert a 6 into the tableau shown in \eqref{tab:ex1}, the insertion process is shown as follows:
\begin{align*}
  \begin{array}{cccc}
    \hat5&\hat5&\hat5&\hat5\\\hat5&\hat5&&\\\hat6&\textcolor{red}{\hat6}&&
  \end{array}
  \to
  \begin{array}{cccc}
    \hat6&\hat6&\hat6&\hat6\\\hat6&\hat6&\textcolor{red}{\hat7}&\\\hat6&\hat7&&
  \end{array}
  \to
  \begin{array}{cccc}
    \hat7&\hat7&\hat7&\hat7\\\hat7&\hat7&\hat8&\textcolor{red}{\hat8}\\\hat7&\hat7&&\\\hat8&\hat8&&
  \end{array}
\end{align*}
where each $\hat i$ denotes a box in $\lambda^i$, and each \textcolor{red}{red} entry denotes an appended box. The resultant tableau is thus
\begin{align*}
  \tilde P=\emptyset\prec2\prec2\prec31\prec41\prec42\prec422\prec432\prec4422=
  \begin{array}{cccc}
    1&1&3&4\\3&5&7&8\\6&6&&\\8&8&&
  \end{array}
\end{align*}

One can visualise the insertion process by building the shapes up vertically.
\begin{center}
  \begin{tikzpicture}
    \tikzstyle{every node}=[fill=white]
    \node (00) at (0,0) {$\emptyset$};
    \node (10) at (3,0) {$\emptyset$};
    \foreach \i in {1,...,3}
    {
    \node (0\i) at (0,\i) {$\lambda^{\i}$};
    \node (1\i) at (3,\i) {$\tilde\lambda^{\i}$};
    }
    \node (04) at (0,4) {$\lambda^{\ell}$};
    \node (14) at (3,4) {$\tilde\lambda^{\ell}$};
    \node at (1.5,3.5) {$\vdots$};
    \foreach \i in {0,...,4}
    \draw[->] (0\i) -- (1\i);
    \foreach \i [count=\j] in {0,...,2}
    {
    \draw[->] (0\i) -- (0\j);
    \draw[->] (1\i) -- (1\j);
    }
    \draw[dotted,->] (03) -- (04);
    \draw[dotted,->] (13) -- (14);
    \draw[->] (0,4.5) --  node[above]{Inserting $k$} (3,4.5);
    \draw[->] (-.5,0) -- node[left]{Tableau $P$} (-.5,4);
    \draw[->] (3.5,0) -- node[right]{Tableau $\tilde P$} (3.5,4);

  \end{tikzpicture}
\end{center}
As we can see, this forms a one-column lattice diagram such that for each $i\le k$, the vertices labeled with $\lambda^{i-1},\lambda^{i},\tilde\lambda^{i-1}$ and $\tilde\lambda^{i}$ surround a box, which we call the $i$th box. 
We can put an $X$ into the $k$th box to indicate that the number inserted into $P$ is $k$. 
The corresponding diagram of the previous example where we insert a 6 into tableau \eqref{tab:ex1} is:
\begin{center}
  \begin{tikzpicture}
    \tikzstyle{every node}=[fill=white]
    \node (00) at (0,0) {$\emptyset$};
    \node (01) at (0,1) {$2$};
    \node (02) at (0,2) {$2$};
    \node (03) at (0,3) {$31$};
    \node (04) at (0,4) {$41$};
    \node (05) at (0,5) {$42$};
    \node (06) at (0,6) {$421$};
    \node (07) at (0,7) {$422$};
    \node (08) at (0,8) {$4322$};

    \node (10) at (1.1,0) {$\emptyset$};
    \node (11) at (1.1,1) {$2$};
    \node (12) at (1.1,2) {$2$};
    \node (13) at (1.1,3) {$31$};
    \node (14) at (1.1,4) {$41$};
    \node (15) at (1.1,5) {$42$};
    \node (16) at (1.1,6) {$422$};
    \node (17) at (1.1,7) {$432$};
    \node (18) at (1.1,8) {$4422$};
    
    \node at (.5,5.55) {$X$};

    \foreach \i in {0,...,8}
    \draw[->] (0\i) -- (1\i);
    \foreach \i [count=\j] in {0,...,7}
    {
    \draw[->] (0\i) -- (0\j);
    \draw[->] (1\i) -- (1\j);
    }

  \end{tikzpicture}
\end{center}
\section{Symmetry property for the Robinson-Schensted algorithm}\label{s:symnormal}
For a word $w=w_1w_2\dots w_n\in[l]^n$, the Robinson-Schensted algorithm starts with the empty tableau $P(0)=\emptyset$.
Then $w_1$ is inserted into $P(0)$ to obtain $P(1)$, then $w_2$ into $P(1)$ to obtain $P(2)$ and so on.
The recording tableau $Q$ is a standard tableau defined by:
\begin{align*}
 Q=\sh P(0)\prec\sh P(1)\prec\dots\prec\sh P(n).
\end{align*}

For example the following table shows the the process of inserting the word $w=31342$:
\begin{center}
  \begin{tabular}{cccccc}
    $i$&1&2&3&4&5\\\hline
    $P(i)$&
    $\begin{array}{c}
      3
    \end{array}$&
    $\begin{array}{cc}
      1&3
    \end{array}$&
    $\begin{array}{cc}
      1&3\\3&
    \end{array}$&
    $\begin{array}{cc}
      1&3\\3&\\4&
    \end{array}$&
    $\begin{array}{ccc}
      1&3&3\\2&&\\4&&
    \end{array}$\\\hline
    $Q^i$&
    $\begin{array}{c}
      3
    \end{array}$&
    $\begin{array}{cc}
      1&2
    \end{array}$&
    $\begin{array}{cc}
      1&2\\3&
    \end{array}$&
    $\begin{array}{cc}
      1&2\\3&\\4&
    \end{array}$&
    $\begin{array}{ccc}
      1&2&5\\3&&\\4&&
    \end{array}$\\
  \end{tabular}.
\end{center}
The corresponding pair of tableaux, which we denote as $(P(w),Q(w))$ are:
\begin{align*}
  (P(w),Q(w))=\left(
    \begin{array}{ccc}
      1&3&3\\2&&\\4&&
    \end{array},
    \begin{array}{ccc}
      1&2&5\\3&&\\4&&
    \end{array}
    \right)
\end{align*}

If we denote $(\lambda^k(i))_{1\le k\le\ell}$ as the shape of the subtableaux for $P(i)$, then since $P(i)$ is obtained from $P(i-1)$ by inserting $w_i$, we can construct a $\{0,1,\dots,n\}\times\{0,1,\dots,\ell\}\subset\mathbb N^2$ lattice growth diagram by concatenating the one-column lattice diagrams defined in the previous section. This is illustrated in the following picture.
\begin{center}
  \begin{tikzpicture}
    \def \s{1.8}
    \tikzstyle{every node}=[fill=white,font=\footnotesize]
    \foreach \i in {0,...,4}
    {
    \node (0\i) at (0,\i*\s) {$\emptyset$};
    \node (\i0) at (\i*\s,0) {$\emptyset$};
    }
    \foreach \i in {1,...,2}
    \foreach \j in {1,...,2}
      \node (\i\j) at (\i*\s,\j*\s) {$\lambda^{\j}(\i)$};

    \foreach \i in {1,...,2}
    {
    \node (\i3) at (\i*\s,3*\s) {$\lambda^{\ell-1}(\i)$};
    \node (\i4) at (\i*\s,4*\s) {$\lambda^{\ell}(\i)$};
    \node (3\i) at (3*\s,\i*\s) {$\lambda^{\i}(n-1)$};
    \node (4\i) at (4*\s,\i*\s) {$\lambda^{\i}(n)$};
    }

    \node (33) at (3*\s,3*\s) {$\lambda^{\ell-1}(n-1)$};
    \node (34) at (3*\s,4*\s) {$\lambda^{\ell}(n-1)$};
    \node (43) at (4*\s,3*\s) {$\lambda^{\ell-1}(n)$};
    \node (44) at (4*\s,4*\s) {$\lambda^{\ell}(n)$};

    \tikzstyle{every node}=[fill=white,font=\huge]
    \node at (\s,2.5*\s) {$\vdots$};
    \node at (3.5*\s,2.5*\s) {$\vdots$};
    \tikzstyle{every node}=[fill=white,font=\small]
    \node at (2.5*\s,\s) {$\ldots$};
    \node at (2.5*\s,3.5*\s) {$\ldots$};
    \foreach \i [count=\ii] in {0,...,1}
    \foreach \j in {0,...,2}
    {
      \draw[->] (\i\j) -- (\ii\j);
      \draw[->] (3\j) -- (4\j);
      \draw[->] (\i3) -- (\ii3);
      \draw[->] (\i4) -- (\ii4);
      \draw[->] (\j\i) -- (\j\ii);
      \draw[->] (\j3) -- (\j4);
      \draw[->] (3\i) -- (3\ii);
      \draw[->] (4\i) -- (4\ii);
    }

    \draw[->] (33) -- (34);
    \draw[->] (33) -- (43);
    \draw[->] (43) -- (44);
    \draw[->] (34) -- (44);
  \end{tikzpicture}
\end{center}
As such, the tableau pair obtained are for $P$ the shapes on the vertices of the rightmost vertical and for $Q$ the shapes on the vertices on the top horizontal line: 
\begin{align*}
P=\lambda^1(n)\nearrow\lambda^2(n)\nearrow\dots\nearrow\lambda^\ell(n);\\
Q=\lambda^\ell(1)\nearrow\lambda^\ell(2)\nearrow\dots\nearrow\lambda^\ell(n).
\end{align*}

For example, if we take $\ell=4$ for word $w=31342$, the growth diagram is as follows:
\begin{center}
  \begin{tikzpicture}
    \draw (0,0) grid (5,4);
    \node at (.5,2.5) {$X$};
    \node at (1.5,.5) {$X$};
    \node at (2.5,2.5) {$X$};
    \node at (3.5,3.5) {$X$};
    \node at (4.5,1.5) {$X$};

    \tikzstyle{every node}=[fill=white]
    \foreach \i in {0,...,5}
    \node (\i0) at (\i,0) {$\emptyset$};
    \foreach \i in {0,...,4}
    \node (0\i) at (0,\i) {$\emptyset$};

    \node at (1,1) {$\emptyset$};
    \node at (1,2) {$\emptyset$};
    \node at (1,3) {$1$};
    \node at (1,4) {$1$};
    \node at (2,1) {$1$};
    \node at (2,2) {$1$};
    \node at (2,3) {$2$};
    \node at (2,4) {$2$};
    \node at (3,1) {$1$};
    \node at (3,2) {$1$};
    \node at (3,3) {$21$};
    \node at (3,4) {$21$};
    \node at (4,1) {$1$};
    \node at (4,2) {$1$};
    \node at (4,3) {$21$};
    \node at (4,4) {$211$};
    \node at (5,1) {$1$};
    \node at (5,2) {$11$};
    \node at (5,3) {$31$};
    \node at (5,4) {$311$};
  \end{tikzpicture}
\end{center}

The Robinson-Schensted algorithm was initially defined to take a permutation as an input and output a pair of \textit{standard} tableaux with the same shape.
In this case we take the word identified by $\sigma(1)\sigma(2)\sigma(3)\dots\sigma(n)$ as the input, which we also denote by $\sigma$. 
A classical result of the algorithm is the symmetry property:
\begin{thm}[See e.g. \cite{sagan,stanley,fulton}]\label{thm:1}For any permutation $\sigma$,
  \begin{align*}
    (P(\sigma^{-1}),Q(\sigma^{-1}))=(Q(\sigma),P(\sigma)).
  \end{align*}
\end{thm}
\begin{proof}[Sketch proof]
We present a growth diagram proof whose row insertion counterpart can be found in \cite{stanley}. Basically the algorithm is reformulated in a way that is symmetric on the $n$ by $n$ growth diagram. We index a box by $(m,k)$ if its four vertices are $(m-1,k-1)$, $(m-1,k)$, $(m,k-1)$ and $(m,k)$.
For any box $(m,k)$ denote by $\lambda,\mu^1,\mu^2,\nu$ the partitions on $(m-1,k-1)$, $(m-1,k)$, $(m,k-1)$ and $(m,k)$ respectively (see the diagram below).
\begin{center}
  \begin{tikzpicture}[scale=1.2]
    \node (100) at (0,0) {$\lambda$};
    \node (110) at (1,0) {$\mu^2$};
    \node (101) at (0,1) {$\mu^1$};
    \node (111) at (1,1) {$\nu$};
    \draw (100) -- (101) -- (111) -- (110) -- (100);
  \end{tikzpicture}
\end{center}
The algorithm goes with the following rule:
\begin{enumerate}
  \item If there's an $X$ in the box and $\lambda=\mu^1=\mu^2$, then $\nu=\lambda+\ve_{l(\lambda)+1}$, that is $\nu$ is obtained by adding a box that forms a new row itself at the bottom of $\lambda$.
  \item If there's no $X$ in the box and $\lambda=\mu^1$, then $\nu=\mu^2$.
  \item If there's no $X$ in the box and $\lambda=\mu^2$, then $\nu=\mu^1$.
  \item If there's no $X$ in the box and $\mu^1=\lambda+\ve_i$, $\mu^2=\lambda+\ve_j$ with $i\neq j$, then $\nu=\lambda+\ve_i+\ve_j=\mu^1\cup\mu^2$.
  \item If there's no $X$ in the box and $\mu^1=\mu^2=\lambda+\ve_i$, then $\nu=\lambda+\ve_i+\ve_{i'}$, where $i'=\max(\{j\le i:\mu^1_{j-1}>\mu^1_j\}\cup\{1\}).$
\end{enumerate}
Note that these rules do not apply to the words case, which is why permutation is special. 
The 5 cases together with a trivial case that belongs to both case 2 and case 3 are illustrated as follows.
\begin{center}
  \begin{tikzpicture}[scale=2.2]
    \node (100) at (0,0) {
    \begin{tikzpicture}[scale=.3]
      \draw (0,0) -- (5,0) -- (5,-1) -- (3,-1) -- (3,-3) -- (2,-3) -- (2,-4) -- (0,-4) -- (0,0);
    \end{tikzpicture}};
    \node (110) at (1,0) {
    \begin{tikzpicture}[scale=.3]
      \draw (0,0) -- (5,0) -- (5,-1) -- (3,-1) -- (3,-3) -- (2,-3) -- (2,-4) -- (0,-4) -- (0,0);
    \end{tikzpicture}};
    \node (101) at (0,1) {
    \begin{tikzpicture}[scale=.3]
      \draw (0,0) -- (5,0) -- (5,-1) -- (3,-1) -- (3,-3) -- (2,-3) -- (2,-4) -- (0,-4) -- (0,0);
    \end{tikzpicture}};
    \node (111) at (1,1) {
    \begin{tikzpicture}[scale=.3]
      \draw (0,0) -- (5,0) -- (5,-1) -- (3,-1) -- (3,-3) -- (2,-3) -- (2,-4) -- (0,-4) -- (0,0);
      \draw (1,-4) -- (1,-5) -- (0,-5) -- (0,-4);
    \end{tikzpicture}};
    \node at (.5,.5) {$X$};
    \node at (.5,-.4) {Case 1};
    \draw (100) -- (101) -- (111) -- (110) -- (100);
    \begin{scope}[shift={(2.2,0)}]
      \node (200) at (0,0) {
      \begin{tikzpicture}[scale=.3]
	\draw (0,0) -- (5,0) -- (5,-1) -- (3,-1) -- (3,-3) -- (2,-3) -- (2,-4) -- (0,-4) -- (0,0);
      \end{tikzpicture}};
      \node (210) at (1,0) {
      \begin{tikzpicture}[scale=.3]
	\draw (0,0) -- (5,0) -- (5,-1) -- (3,-1) -- (3,-3) -- (2,-3) -- (2,-4) -- (0,-4) -- (0,0);
	\draw (3,-3) -- (3,-4) -- (2,-4);
      \end{tikzpicture}};
      \node (201) at (0,1) {
      \begin{tikzpicture}[scale=.3]
	\draw (0,0) -- (5,0) -- (5,-1) -- (3,-1) -- (3,-3) -- (2,-3) -- (2,-4) -- (0,-4) -- (0,0);
      \end{tikzpicture}};
      \node (211) at (1,1) {
      \begin{tikzpicture}[scale=.3]
	\draw (0,0) -- (5,0) -- (5,-1) -- (3,-1) -- (3,-3) -- (2,-3) -- (2,-4) -- (0,-4) -- (0,0);
	\draw (3,-3) -- (3,-4) -- (2,-4);
      \end{tikzpicture}};
      \draw (200) -- (201) -- (211) -- (210) -- (200);
      \node at (.5,-.4) {Case 2};
    \end{scope}
  \end{tikzpicture}
\end{center}
\begin{center}
  \begin{tikzpicture}[scale=2.2]
    \node (100) at (0,0) {
    \begin{tikzpicture}[scale=.3]
      \draw (0,0) -- (5,0) -- (5,-1) -- (3,-1) -- (3,-3) -- (2,-3) -- (2,-4) -- (0,-4) -- (0,0);
    \end{tikzpicture}};
    \node (110) at (1,0) {
    \begin{tikzpicture}[scale=.3]
      \draw (0,0) -- (5,0) -- (5,-1) -- (3,-1) -- (3,-3) -- (2,-3) -- (2,-4) -- (0,-4) -- (0,0);
    \end{tikzpicture}};
    \node (101) at (0,1) {
    \begin{tikzpicture}[scale=.3]
      \draw (0,0) -- (5,0) -- (5,-1) -- (3,-1) -- (3,-3) -- (2,-3) -- (2,-4) -- (0,-4) -- (0,0);
      \draw (5,0) -- (6,0) -- (6,-1) -- (5,-1);
    \end{tikzpicture}};
    \node (111) at (1,1) {
    \begin{tikzpicture}[scale=.3]
      \draw (0,0) -- (5,0) -- (5,-1) -- (3,-1) -- (3,-3) -- (2,-3) -- (2,-4) -- (0,-4) -- (0,0);
      \draw (5,0) -- (6,0) -- (6,-1) -- (5,-1);
    \end{tikzpicture}};
    \draw (100) -- (101) -- (111) -- (110) -- (100);
    \node at (.5,-.5) {Case 3};
    \begin{scope}[shift={(2.2,0)}]
      \node (200) at (0,0) {
      \begin{tikzpicture}[scale=.3]
	\draw (0,0) -- (5,0) -- (5,-1) -- (3,-1) -- (3,-3) -- (2,-3) -- (2,-4) -- (0,-4) -- (0,0);
      \end{tikzpicture}};
      \node (210) at (1,0) {
      \begin{tikzpicture}[scale=.3]
	\draw (0,0) -- (5,0) -- (5,-1) -- (3,-1) -- (3,-3) -- (2,-3) -- (2,-4) -- (0,-4) -- (0,0);
	\draw (1,-4) -- (1,-5) -- (0,-5) -- (0,-4);
      \end{tikzpicture}};
      \node (201) at (0,1) {
      \begin{tikzpicture}[scale=.3]
	\draw (0,0) -- (5,0) -- (5,-1) -- (3,-1) -- (3,-3) -- (2,-3) -- (2,-4) -- (0,-4) -- (0,0);
	\draw (4,-1) -- (4,-2) -- (3,-2);
      \end{tikzpicture}};
      \node (211) at (1,1) {
      \begin{tikzpicture}[scale=.3]
	\draw (0,0) -- (5,0) -- (5,-1) -- (3,-1) -- (3,-3) -- (2,-3) -- (2,-4) -- (0,-4) -- (0,0);
	\draw (1,-4) -- (1,-5) -- (0,-5) -- (0,-4);
	\draw (4,-1) -- (4,-2) -- (3,-2);
      \end{tikzpicture}};
      \draw (200) -- (201) -- (211) -- (210) -- (200);
      \node at (.5,-.5) {Case 4};
    \end{scope}
    \begin{scope}[shift={(0,-2.2)}]
    \node (100) at (0,0) {
    \begin{tikzpicture}[scale=.3]
      \draw (0,0) -- (5,0) -- (5,-1) -- (3,-1) -- (3,-3) -- (2,-3) -- (2,-4) -- (0,-4) -- (0,0);
    \end{tikzpicture}};
    \node (110) at (1,0) {
    \begin{tikzpicture}[scale=.3]
      \draw (0,0) -- (5,0) -- (5,-1) -- (3,-1) -- (3,-3) -- (2,-3) -- (2,-4) -- (0,-4) -- (0,0);
      \draw (3,-3) -- (3,-4) -- (1,-4);
    \end{tikzpicture}};
    \node (101) at (0,1) {
    \begin{tikzpicture}[scale=.3]
      \draw (0,0) -- (5,0) -- (5,-1) -- (3,-1) -- (3,-3) -- (2,-3) -- (2,-4) -- (0,-4) -- (0,0);
      \draw (3,-3) -- (3,-4) -- (1,-4);
    \end{tikzpicture}};
    \node (111) at (1,1) {
    \begin{tikzpicture}[scale=.3]
      \draw (0,0) -- (5,0) -- (5,-1) -- (3,-1) -- (3,-3) -- (2,-3) -- (2,-4) -- (0,-4) -- (0,0);
      \draw (3,-3) -- (3,-4) -- (1,-4);
      \draw (4,-1) -- (4,-2) -- (3,-2);
    \end{tikzpicture}};
    \draw (100) -- (101) -- (111) -- (110) -- (100);
    \node at (.5,-.5) {Case 5};
    \end{scope}[shift={(0,-2.2)}]
    \begin{scope}[shift={(2.2,-2.2)}]
      \node (200) at (0,0) {
      \begin{tikzpicture}[scale=.3]
	\draw (0,0) -- (5,0) -- (5,-1) -- (3,-1) -- (3,-3) -- (2,-3) -- (2,-4) -- (0,-4) -- (0,0);
      \end{tikzpicture}};
      \node (210) at (1,0) {
      \begin{tikzpicture}[scale=.3]
	\draw (0,0) -- (5,0) -- (5,-1) -- (3,-1) -- (3,-3) -- (2,-3) -- (2,-4) -- (0,-4) -- (0,0);
      \end{tikzpicture}};
      \node (201) at (0,1) {
      \begin{tikzpicture}[scale=.3]
	\draw (0,0) -- (5,0) -- (5,-1) -- (3,-1) -- (3,-3) -- (2,-3) -- (2,-4) -- (0,-4) -- (0,0);
      \end{tikzpicture}};
      \node (211) at (1,1) {
      \begin{tikzpicture}[scale=.3]
	\draw (0,0) -- (5,0) -- (5,-1) -- (3,-1) -- (3,-3) -- (2,-3) -- (2,-4) -- (0,-4) -- (0,0);
      \end{tikzpicture}};
      \draw (200) -- (201) -- (211) -- (210) -- (200);
      \node at (.5,-.55) {\begin{tabular}{c}``Boring'' intersection \\ of Case 2 and Case 3\end{tabular}};
    \end{scope}
  \end{tikzpicture}
\end{center}
This way, all the vertices of the diagram can be labelled recursively given $\emptyset$ as the boundary condition on the leftmost and bottom vertices. One could check that this is indeed equivalent to the definition in the previous section.
Moreover, the transposition of the lattice diagram preserves the algorithm. 
That is: if we put $X$'s in boxes $(\sigma(i),i)_{i\le n}$ rather than $(i,\sigma(i))_{i\le n}$, and label each vertex $(i,j)$ with what was labelled on $(j,i)$, we end up with the configuration that is the same as if we apply the rules 1 through 5.
This immediately finishes the proof.
\end{proof}

\section{A $q$-weighted Robinson-Schensted algorithm}\label{s:qrs}
In a recent paper \cite{op13} a $q$-weighted Robinson-Schensted algorithm was introduced. In this and the next section we describe the algorithm in a different way from the definition in \cite{op13}. At the end of next section it is obvious to see that:
\begin{prop}
  The algorithm described in this section for inserting a letter to a tableau and in the next section for inserting a word to an empty tableau is an equivalent reformulation of the $q$-weighted Robinson-Schensted algorithm defined in \cite{op13}.
\end{prop}
When inserting a letter to a tableau it outputs a weigted set of tableaux. 
To insert a $k$ into $P$, we start with $\lambda^k$, append a box to different possible rows no lower than $k$, record the index of rows, and for each one of these indices $j_k$, we obtain a new $k$th shape $\tilde\lambda^k=\lambda^k+\ve_{j_k}$ with weight $w_0(k,j_k)$; then we proceed to add a box to all possible rows in $\lambda^{k+1}$ no lower than $j_k$ and obtain a weighted set of new $k+1$th shapes $\{(\tilde\lambda^{k+1}=\lambda^{k+1}+\ve_{j_{k+1}},w_1(k+1,j_{k+1})):j_{k+1}\le j_k\}$; then for each $j_{k+1}$ we obtain the new $k+2$th shapes with weight $w_1(k+2,j_{k+2})$ by adding a box to $j_{k+2}$th row in $\lambda^{k+2}$ and so on and so forth. 
We also prune all the 0-weighted branches.

This way we obtain a forest of trees whose roots are $\tilde\lambda^k$'s, leaves are $\tilde\lambda^\ell$'s with edges labeled by the weights.
If we prepend $\emptyset\prec\lambda^1\prec\lambda^2\prec\dots\prec\lambda^{k-1}$ to this forest we obtain a tree with root $\emptyset$, the first $k$ levels each having one branch with one edge, which we label with weight 1. 
Then for each leaf $\tilde\lambda^\ell$ we obtain a unique tableau by reading its genealogy. We also associate this tableau with a weight which is the product of weights along the edges. 
By going through all leaves we obtain a set of weighted tableaux, which is the output of $q$-inserting a $k$ to $P$.

Now we describe how we calculate the weights when inserting a box to $i$th shape $\lambda^i$ with a $j_{i-1}$ (let $j_{k-1}=k$) specified.
First let us define functions $f_0$ and $f_1$ associated with a pair of nested partitions $\mu\prec\lambda$:
\begin{align*}
  f_0(j;\mu,\lambda)=1-q^{\mu_{j-1}-\lambda_j};\qquad f_1(j;\mu,\lambda)=\frac{1-q^{\mu_{j-1}-\lambda_j}}{1-q^{\mu_{j-1}-\mu_j}};\quad 1<j\le k\\
  f_0(1;\mu,\lambda)=f_1(1;\mu,\lambda)=1.
\end{align*}
When adding a box to the first affected partition $\lambda^k$, the weight $w_0(k,j_k)$ associated with each $j_k\le k$ is
\begin{align*}
  w_0(k,j_k)=f_0(j_k;\lambda^{k-1},\lambda^k)\prod_{p=j_k+1}^k(1-f_0(p;\lambda^{k-1},\lambda^k)).
\end{align*}
For $r>k$, and for a fixed pair of $(\lambda^r,j_{r-1})$, when adding a box to row $j_r\le j_{r-1}$ of $\lambda^{r}$, the corresponding weight $w_1(r,j_r)$ is
\begin{align*}
  w_1&(r,j_r)\\
  &=\begin{cases}
    f_0(j_r;\lambda^{r-1},\lambda^r)\left(\prod_{p=j_r+1}^{j_{r-1}-1}(1-f_0(p;\lambda^{r-1},\lambda^r))\right)\\\times(1-f_1(j_{r-1};\lambda^{r-1},\lambda^r)),&1\le j_r<j_{r-1}\\
    f_1(j_r;\lambda^{r-1},\lambda^r),&j_r=j_{r-1}.
  \end{cases}
\end{align*}

For example, below is a path of the tree, i.e. genealogy of a possible output tableau when inserting a $5$ into tableau \eqref{tab:ex1}:
\begin{align*}
  \emptyset\xrightarrow{1}
  \begin{array}{cc}
    \hat1&\hat1
  \end{array}\xrightarrow{1}
  \begin{array}{cc}
    \hat1&\hat1
  \end{array}\xrightarrow{1}
  \begin{array}{ccc}
    \hat2&\hat2&\hat3\\\hat3&&
  \end{array}\xrightarrow{1}
  \begin{array}{cccc}
    \hat3&\hat3&\hat3&\hat4\\\hat3&&&
  \end{array}\xrightarrow{1-q}
  \begin{array}{cccc}
    \hat4&\hat4&\hat4&\hat4\\\hat4&\hat5&&\\\textcolor{red}{\hat5}&&&
  \end{array}\\
  \xrightarrow{\left(1-\frac{1-q}{1-q^2}\right)(1-q^2)}
  \begin{array}{cccc}
    \hat5&\hat5&\hat5&\hat5\\\hat5&\hat5&\textcolor{red}{\hat6}&\\\hat6&&&
  \end{array}\xrightarrow{1}
  \begin{array}{cccc}
    \hat6&\hat6&\hat6&\hat6\\\hat6&\hat6&\textcolor{red}{\hat7}&\\\hat6&\hat7&&
  \end{array}\xrightarrow{\frac{1-q}{1-q^2}}
  \begin{array}{cccc}
    \hat7&\hat7&\hat7&\hat7\\\hat7&\hat7&\hat8&\textcolor{red}{\hat8}\\\hat7&\hat7&&\\\hat8&\hat8&&
  \end{array}
\end{align*}
where again each $\hat r$ denotes a box in $\lambda^r$ and the \textcolor{red}{red} $\textcolor{red}{\hat r}$ denotes the new box $\tilde\lambda^r/\lambda^r$. We also have $j_5=3$, $j_6=j_7=j_8=2$ in this example.
The above output tableau is
\begin{align*}
  \tilde P=\emptyset\prec2\prec2\prec31\prec41\prec421\prec431\prec432\prec4422=
  \begin{array}{cccc}
    1&1&3&4\\3&5&6&8\\5&7&&\\8&8&&
  \end{array}
\end{align*}
with weight $1\cdot1\cdot1\cdot1\cdot(1-q)\cdot\left(1-\frac{1-q}{1-q^2}\right)(1-q^2)\cdot1\cdot\frac{1-q}{1-q^2}=\frac{q(1-q)^2}{1+q}$. For any two tableaux $P$ and $\tilde P$ we denote the weight of obtaining $\tilde P$ after inserting $k$ into $P$ by $I_k(P,\tilde P)$. So in the previous example we have $I_5(P,\tilde P)=\frac{q(1-q)^2}{1+q}$.

As we can see, each branching to obtain a weighted set of new $i$th shapes $\{\tilde\lambda^i\}$ is determined by $\lambda^i$ and $j_{i-1}$, the latter in turn is identified by the pair $(\lambda^{i-1},\tilde\lambda^{i-1})$. 
Therefore the branching is determined by the triplet $(\lambda^{i-1},\lambda^i,\tilde\lambda^{i-1})$. While each $\tilde\lambda^i$, together with $\lambda^i$ and $\lambda^{i+1}$ determines a next branching. 
Hence the tree has the structure illustrated in Figure \ref{fig:branch},
\begin{figure}[h!]
  \includegraphics[scale=.3]{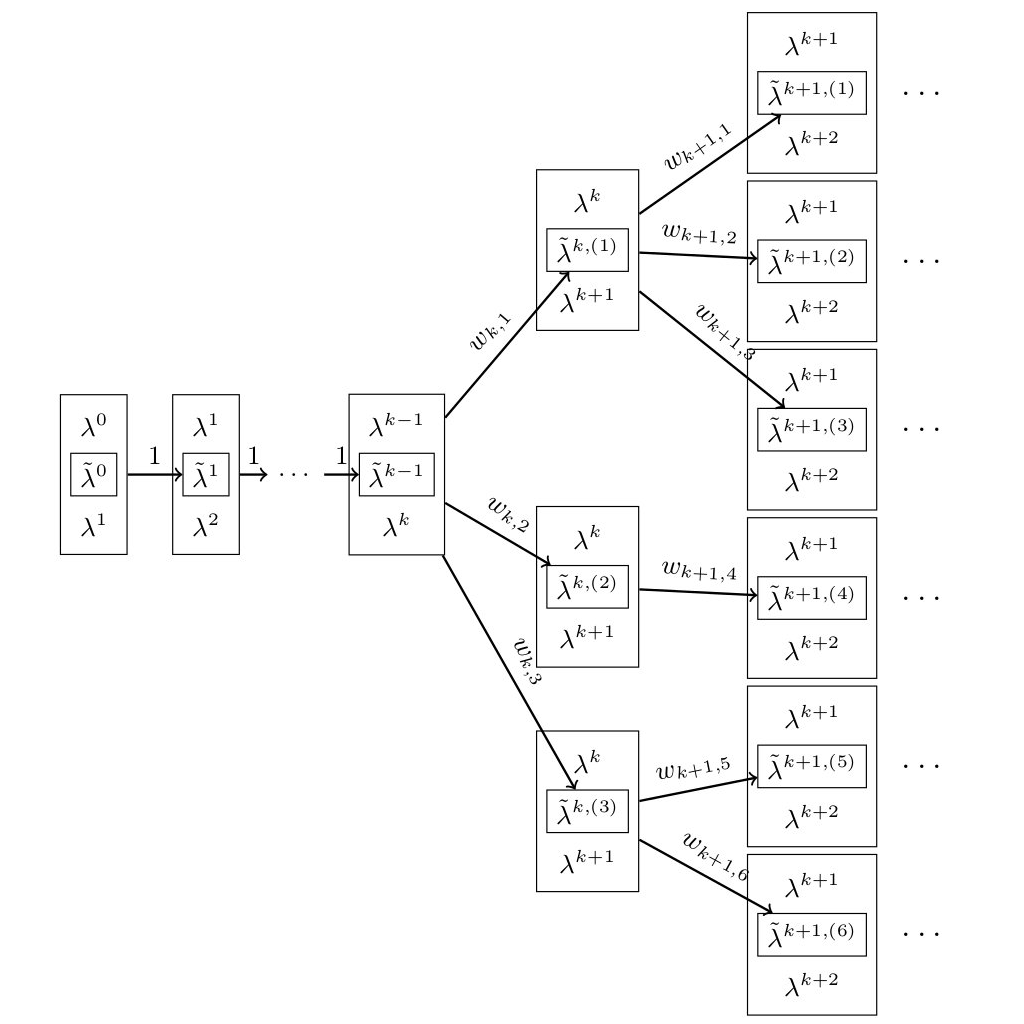}
  \caption{The structure of a branching insertion algorithm when inserting a letter $k$ (see Section \ref{s:other}, where the $q$-weighted algorithm is an example).}
  \label{fig:branch}
\end{figure}
where we put the weights on the edges. This way we can write the weights:
\begin{align*}
w((\lambda^{i-1},\lambda^i,\tilde\lambda^{i-1}),\tilde\lambda^i) =
\begin{cases}
  \mathbb I_{\tilde\lambda^i=\lambda^i} &i<k\\
  w_0(i,j_i)&i=k\\
  w_1(i,j_i)&i>k
\end{cases}
\end{align*}
where $\tilde\lambda^i =\lambda^i+\ve_{j_i}$, $j_i\le j_{i-1}$ for $i\ge k$, and $\mathbb I$ is the indicator function.

Therefore, the $q$-insertion algorithm can be visualised in a ``one-column'' graph similar to the normal insertion.
Each node $\lambda^{i}$ (or $\tilde\lambda^{j,(p_j)}$) representing the old $i$th shape (or a new $j$th shape) is associated with a vertex $(0,i)$ (or $(1,j)$) as the northwest (or northeast) vertex of the $i$th (or $j$th) box.
Each triplet $(\lambda^{i-1},\lambda^i,\tilde\lambda^{i-1,(p_{i-1})})$ is represented as a connected triple nodes surrounding the $i$th box from the south and the west, 
for which each branch $\tilde\lambda^{i,(p_i)}$ is represented as a node that connects to both $\lambda^i$ and $\tilde\lambda^{i-1,(p_{i-1})}$, 
where we put the weights on both the horizontal and vertical edges. 
\begin{align*}
  w(\tilde\lambda^{i-1,(p_{i-1})}\to \tilde \lambda^{i,(p_i)})=w(\lambda^{i}\to \tilde \lambda^{i,(p_i)})=w((\lambda^{i-1},\lambda^i,\tilde\lambda^{i-1,(p_{i-1})}),\tilde\lambda^i),\quad i\ge1.
\end{align*}
As is in the normal insertion case, we put an $X$ in the $k$th box to indicate we are inserting a $k$.
Figure \ref{fig:insertletter} is such a visualisation corresponding to the tree in Figure \ref{fig:branch}. 
\begin{figure}[h!]
  \centering
  \begin{tikzpicture}[scale=.9]
    \def \s{1}
    \def \c{6}
    \def \d{9}
    \def \e{14}
    \def \f{1.5}
    \tikzstyle{every node}=[draw,rectangle]
    \node (00) at (0,0) {$\emptyset$};
    \node (10) at (\c,0) {$\emptyset$};
    \node (01) at (0,\f) {$\lambda^1$};
    \node (11) at (\c,\f) {$\tilde\lambda^1=\lambda^1$};
    \node (02) at (0,2*\f) {$\lambda^2$};
    \node (12) at (\c,2*\f) {$\tilde\lambda^2=\lambda^2$};
    \node (03) at (0,3*\f) {$\lambda^{k-1}$};
    \node (13) at (\c,3*\f) {$\tilde\lambda^{k-1}=\lambda^{k-1}$};
    \node (04) at (0,\d) {$\lambda^{k}$};
    \node (143) at (\c-\s,\d-\s) {$\tilde\lambda^{k,(3)}$};
    \node (142) at (\c,\d) {$\tilde\lambda^{k,(2)}$};
    \node (141) at (\c+\s,\d+\s) {$\tilde\lambda^{k,(1)}$};
    \node (05) at (0,\e) {$\lambda^{k+1}$};
    \node (156) at (\c-2*\s,\e-2*\s) {$\tilde\lambda^{k+1,(6)}$};
    \node (155) at (\c-1*\s,\e-1*\s) {$\tilde\lambda^{k+1,(5)}$};
    \node (154) at (\c,\e) {$\tilde\lambda^{k+1,(4)}$};
    \node (153) at (\c+\s,\e+\s) {$\tilde\lambda^{k+1,(3)}$};
    \node (152) at (\c+2*\s,\e+2*\s) {$\tilde\lambda^{k+1,(2)}$};
    \node (151) at (\c+3*\s,\e+3*\s) {$\tilde\lambda^{k+1,(1)}$};

    \tikzstyle{every node}=[sloped,fill=white,draw=none,font=\small]
    \foreach \i in {0,...,3}
      \draw (0\i) -- node{$1$} (1\i);
    \foreach \i [count=\ii] in {0,...,4}
    {
    \ifthenelse{\NOT 2 = \i}{\draw (0\i) -- (0\ii);}{}
    }
    \foreach \i [count=\ii] in {0,...,1}
      \draw (1\i) -- node{$1$} (1\ii);
    \draw[dotted,thick] (02) -- (03);
    \draw[dotted,thick] (12) -- (13);
    \draw (13) -- node{$w_{k,1}$} (141);
    \draw (13) -- node{$w_{k,2}$} (142);
    \draw (13) -- node{$w_{k,3}$} (143);
    \draw (141) -- node{$w_{k+1,1}$} (151);
    \draw (141) -- node{$w_{k+1,2}$} (152);
    \draw (141) -- node{$w_{k+1,3}$} (153);
    \draw (142) -- node{$w_{k+1,4}$} (154);
    \draw (143) -- node{$w_{k+1,5}$} (155);
    \draw (143) -- node{$w_{k+1,6}$} (156);
    \draw (04) -- node{$w_{k,1}$} (141);
    \draw (04) -- node{$w_{k,2}$} (142);
    \draw (04) -- node{$w_{k,3}$} (143);
    \draw (05) -- node{$w_{k+1,1}$} (151);
    \draw (05) -- node{$w_{k+1,2}$} (152);
    \draw (05) -- node{$w_{k+1,3}$} (153);
    \draw (05) -- node{$w_{k+1,4}$} (154);
    \draw (05) -- node{$w_{k+1,5}$} (155);
    \draw (05) -- node{$w_{k+1,6}$} (156);

    \node at (3,6.5) [font=\Huge]{$X$};
  \end{tikzpicture}
  \caption{The one-column construction of $q$-inserting a letter to a tableau.}
  \label{fig:insertletter}
\end{figure}
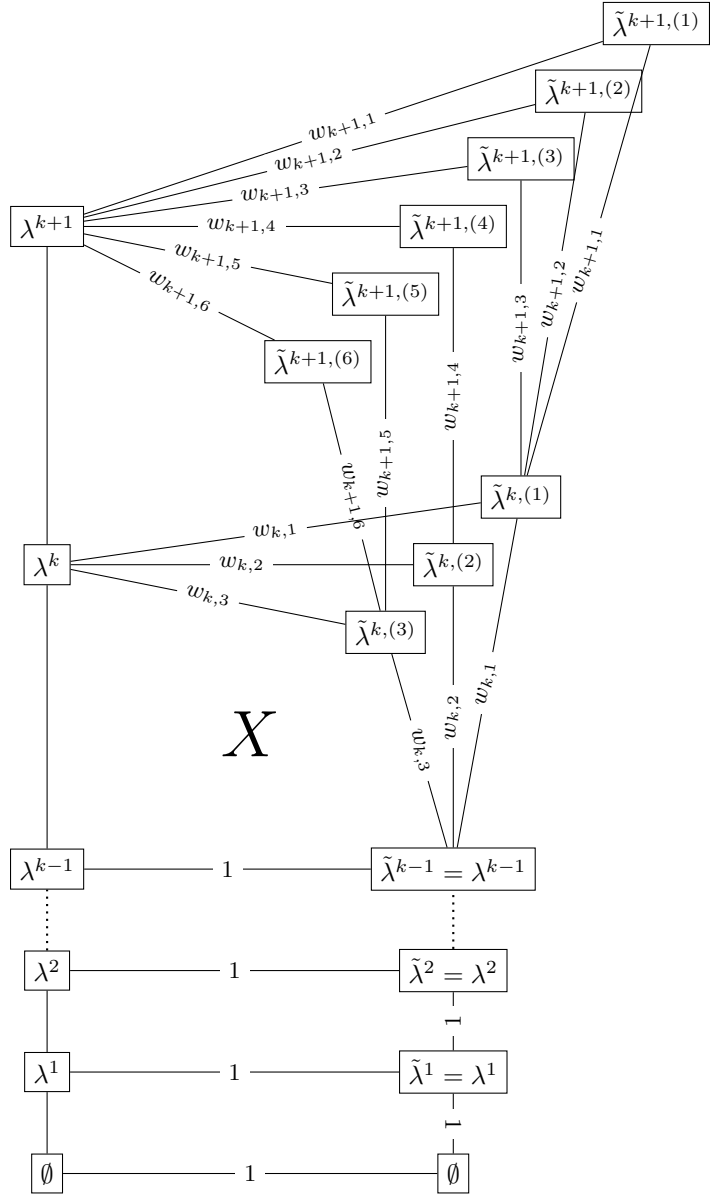
All the nodes associated with the right column $((i,0):0\le i\le\ell)$ form a tree, such that each $\tilde P$ in the output corresponds to a genealogy of a node associated with $(1,\ell)$, whose weight $I_k(P,\tilde P)$ can be obtained as the product of the weights along the genealogical line of the tree.
\section{Word input for the $q$-weighted Robinson-Schensted algorithm}\label{s:wordqrs}
We can also take a word as input for the $q$-weighted Robinson-Schensted algorithm and produce a set of weighted pairs of tableaux. 
We start with the set of only one pair of tableaux $(P(0),Q^0)=(\emptyset,\emptyset)$ with weight 1.
Suppose at time $m$, we have obtained a set of weighted pairs of tableaux 
$$\{(P(m)^{(1)},(Q^m)^{(1)},w(m)^{(1)}),\cdots,(P(m)^{(p)},(Q^m)^{(p)},w(m)^{(p)})\}.$$ We use brackets in superscripts to indicate the different possibilities, in order not to confuse with subtableaux or shapes.
We take each triplet \\$(P(m)^{(i)},(Q^m)^{(i)},w(m)^{(i)})$ in the collection, insert $w_{m+1}$ into $P(m)^{(i)}$ to produce a new set of $P$-tableaux paired with the weight generated during the insertion
$$\{(P(m+1)^{(i,1)},w^{(i,1)}),(P(m+1)^{(i,2)},w^{(i,2)}),\cdots,(P(m+1)^{(i,r_i)},w^{(i,r_i)})\}.$$
Then for each $P(m+1)^{(i,j)}$, $1\le j\le r_i$, we pair it with a $(Q^{m+1})^{(i,j)}$ by adding a box with entry $m+1$ to $(Q^m)^{(i)}$ such that $\sh P(m+1)^{(i,j)}=\sh (Q^{m+1})^{(i,j)}$. We also attach a weight $w(m+1)^{(i,j)}=w(m)w^{(i,j)}$ to this pair.
This way, for each triplet $(P(m)^{(i)},(Q^m)^{(i)},w(m)^{(i)})$ we have obtained its branching set 
\begin{align*}
&\{(P(m+1)^{(i,1)},(Q^{m+1})^{(i,1)},w(m+1)^{(i,1)}),\\
&\qquad(P(m+1)^{(i,2)},(Q^{m+1})^{(i,2)},w(m+1)^{(i,2)}), \\
&\qquad\qquad\qquad\qquad\qquad\cdots,(P(m+1)^{(i,r_i)},(Q^{m+1})^{(i,r_i)},w(m+1)^{(i,r_i)})\}.
\end{align*}
Let $i$ run over $1,2,\cdots,p$, we obtained a collection of all possible branchings:
$$((P(m+1)^{(i,j)},(Q^{m+1})^{(i,j)},w(m+1)^{(i,j)}))_{1\le i\le p,1\le j\le r_i}.$$
Note by \textit{collection} we mean a vector of objects that allows repeats, as opposed to a \textit{set}.
We merge triplets with the same tableau pair by adding up their weights. This way we have obtained the weighted set of pairs of tableaux at time $m+1$.

The output of inserting the word $w_1w_2\dots w_n$ is defined as the weighted set at time $n$. We denote by $\phi_w(P,Q)$ the weight of pair $(P,Q)$ in this set.
This is defined recursively by the following formula (see \cite{op13}). For $P$ and $Q$ with the same shape of size $n$ and word $w$ of length $n-1$,
\begin{align*}
  \phi_{wk}(P,Q)=\sum\phi_w(\hat P,Q^{n-1})I_k(\hat P,P),
\end{align*}
where the sum is over all tableau $\hat P$ with the same shape as $Q^{n-1}$.

This, like in the classical case, can be visualised as a graph whose nodes are associated with vertices on $D_{n,\ell}:=\{0,\dots,n\}\times \{0,\dots,\ell\}$ lattice diagram by concatenating the graphs of the one-column construction in Figure \ref{fig:insertletter} associated with the insertion of each letter. Each possible $\lambda^{k,(p_{m,k})}(m)$ is associated with the vertex $(m,k)$.
For example, we obtain the following graph if we apply the algorithm to $2132\in[3]^4$, where the unlabeled edges have weight 1, which will be the case hereafter.
\begin{center}
  \begin{tikzpicture}[scale=.9]
    \def \lh{{0,1.5,3,5.5,11}}
    \def \lv{{0,1.5,4,9.5}}
    \def \s{1}

    \tikzstyle{every node}=[draw]
    \foreach \i in {0,...,4}
    \node (\i0) at (\lh[\i],\lv[0]) {$\emptyset$};
    \foreach \i in {0,...,3}
    \node (0\i) at (\lh[0],\lv[\i]) {$\emptyset$};
    \node (11) at (\lh[1],\lv[1]) {$\emptyset$};
    \foreach \i in {2,...,4}
    \node (\i1) at (\lh[\i],\lv[1]) {$1$};
    \foreach \i in {2,...,3}
    {
    \node (1\i) at (\lh[1],\lv[\i]) {$1$};
    \node (2\i) at (\lh[2],\lv[\i]) {$2$};
    }
    \node (32) at (\lh[3],\lv[2]) {$2$};
    \node (331) at (\lh[3],\lv[3]) {$21$};
    \node (332) at (\lh[3]-\s,\lv[3]-\s) {$3$};
    \node (421) at (\lh[4],\lv[2]) {$21$};
    \node (422) at (\lh[4]-\s,\lv[2]-\s) {$3$};
    \node (431) at (\lh[4]+\s,\lv[3]+\s) {$22$};
    \node (432) at (\lh[4]-0.3*\s,\lv[3]-0.3*\s) {$31$};
    \node (433) at (\lh[4]-\s,\lv[3]-\s) {$31$};
    \node (434) at (\lh[4]-2*\s,\lv[3]-2*\s) {$31$};
    \node (435) at (\lh[4]-3*\s,\lv[3]-3*\s) {$4$};

    \tikzstyle{every node}=[]
    \node at (\lh[0]/2+\lh[1]/2,\lv[1]/2+\lv[2]/2) {$X$};
    \node at (\lh[1]/2+\lh[2]/2,\lv[0]/2+\lv[1]/2) {$X$};
    \node at (\lh[2]/2+\lh[3]/2,\lv[2]/2+\lv[3]/2) {$X$};
    \node at (\lh[3]/2+\lh[4]/2,\lv[1]/2+\lv[2]/2) {$X$};

    \tikzstyle{every node}=[fill=white,sloped,font=\small]
    \foreach \i [count=\ii] in {0,...,3}
    {
      \draw (\i0) --  (\ii0);
      \draw (\i1) --  (\ii1);
    }
    \foreach \i [count=\ii] in {0,...,2}
    {
      \draw (0\i) --  (0\ii);
      \draw (1\i) --  (1\ii);
      \draw (2\i) --  (2\ii);
    }
    \foreach \i [count=\ii] in {0,1}
    {
      \draw (\i2) --  (\ii2);
      \draw (\i3) --  (\ii3);
    }
    \draw (30) --  (31);
    \draw (40) --  (41);
    \draw (31) --  (32);
    \draw (22) --  (32);
    \draw (32) -- node{$1-q^2$} (331);
    \draw (32) -- node{$q^2$} (332);
    \draw (23) -- node{$1-q^2$} (331);
    \draw (23) -- node{$q^2$} (332);
    \draw (41) -- node{$1-q$} (421);
    \draw (41) -- node{$q$} (422);
    \draw (32) -- node{$1-q$} (421);
    \draw (32) -- node{$q$} (422);
    \draw (421) -- node{$(1-q)/(1-q^2)$} (431);
    \draw (421) -- node{$1-(1-q)/(1-q^2)$} (432);
    \draw (331) -- node{$(1-q)/(1-q^2)$} (431);
    \draw (331) -- node{$1-(1-q)/(1-q^2)$} (432);
    \draw (421) --  (434);
    \draw (332) --  (434);
    \draw (422) --  (433);
    \draw (331) --  (433);
    \draw (422) --  (435);
    \draw (332) --  (435);

  \end{tikzpicture}
\end{center}
From the definition of the algorithm, each node $\lambda^{\ell,(p)}(n)$ at $(n,\ell)$ corresponds to a $(P,Q)$ tableau pair in the output set. 
Indeed we can trace back the genealogy of $\lambda^{\ell,(p)}(n)$ and find a unique array of indices $(p_{m,k})_{0\le m\le n,0\le k\le \ell}$ with $p_{n,\ell}=p$ such that nodes $(\lambda^{k,(p_{m,k})}(m))_{m,k}$ are connected and form a diagram that is isomorphic to the lattice diagram $D_{n,\ell}$.
Clearly,
\begin{align*}
  P^{(p)}=\lambda^{0,(p_{n,0})}(n)\prec\lambda^{1,(p_{n,1})}(n)\prec\dots\prec\lambda^{\ell,(p_{n,\ell})}(n);\\
  Q^{(p)}=\lambda^{\ell,(p_{0,\ell})}(0)\prec\lambda^{\ell,(p_{1,\ell})}(1)\prec\dots\prec\lambda^{\ell,(p_{n,\ell})}(n),
\end{align*}
are the corresponding $(P,Q)$ tableau pair. Moreover, we can identify a weight with the node $\lambda^{\ell,(p)}$ by multiplying the weights of all the horizontal (or all the vertical) edges along its genealogical diagram and denote it by $w(\lambda^{\ell,(p)})$:
\begin{align*}
  w(\lambda^{\ell,(p)}(n))=\prod_{1\le m\le n, 0\le k\le \ell} w(\lambda^{k,(p_{m-1,k})}(m-1)\to\lambda^{k,(p_{m,k})}(m))\\
  =\prod_{0\le m\le n, 1\le k\le \ell} w(\lambda^{k-1,(p_{m,k-1})}(m)\to\lambda^{k,(p_{m,k-1})}(m)).
\end{align*}
Then by definition,
\begin{align*}
  \phi_w(P,Q)=\sum w(\lambda^{\ell,(p)}(n))\mathbb I_{P^{(p)}=P,Q^{(p)}=Q},
\end{align*}
where the sum is over all nodes $\lambda^{l,(p)}(n)$ at vertex $(n,l)$.
\section{The symmetry property for the $q$-weighted RS algorithm with permutation input}\label{s:symq}
When we take a permutation input $\sigma$ which is identified as a word with distinct letters in the same way as in the normal RS algorithm, we also end up with a symmetry property, which is the main result of this paper:
\begin{thm}\label{thm:main}
For any permutation $\sigma\in S_n$ and standard tableau pair $(P,Q)$,
\begin{align*}
  \phi_{\sigma^{-1}}(Q,P)=\phi_\sigma(P,Q).
\end{align*}
\end{thm}
As in the normal RS algorithm case, permutation input are special in that we can restate the insertion rule in a symmetric fashion. For each box we pick a connected triplet on southwest, northwest and southeast corners. 
That is, suppose the box has index $(m,k)$ we fix arbitrary $p_1$, $p_2$ and $p_3$ such that $(\lambda,\mu^1,\mu^2):=(\lambda^{k-1,(p_1)}(m-1),\lambda^{k,(p_2)}(m-1),\lambda^{k-1,(p_3)}(m))$ is a connected triplet. 
We denote by $N$ the set of partitions on the vertex $(m,k)$ that are connected to $\mu^1$ and $\mu^2$. And for any $\nu\in N$, we write $w(\mu,\nu)$ instead of $w(\mu^1\to\nu)$ or $w(\mu^2\to\nu)$ since they are equal, and for the sake of symmetry.
Define an operator $I^k:W\to W$ by
\begin{align*}
  I^k\lambda:=\lambda+\ve_{\max\{j\le k:\lambda+\ve_j\in W\}}.
\end{align*}
and a set $\Lambda^k(\lambda)\subset W$ by
\begin{align*}
  \Lambda^k(\lambda):=\{I^j\lambda:j\le k\},
\end{align*}
and denote $\Lambda(\lambda):=\Lambda^{l(\lambda)+1}(\lambda)$.

Then $N$ and $(w(\mu,\nu):\nu\in N)$ belongs to one of the following 5 cases.
\begin{enumerate}
  \item The box has an $X$ in it, and $\mu^1$, $\mu^2$ are equal to $\lambda$. Then $N$ consists of all possible partitions obtained by adding a box to $\lambda$:
    \begin{align*}
      N=\Lambda(\lambda).
    \end{align*}
    The weights are
    \begin{align*}
      w(\mu,\nu)=\begin{cases}
	q^{\lambda_j}-q^{\lambda_{j-1}},&\text{if }\nu=\lambda+\ve_j\text{ for some }j>1\\
	q^{\lambda_1},&\text{if }\nu=\lambda+\ve_1
      \end{cases},\quad \nu\in N.
    \end{align*}
  \item The box does not have an $X$ in it and $\mu^1=\lambda$. Then $N$ is a singleton which is the same as $\mu^2$,
    and the weight is 1:
    \begin{align*}
      w(\mu,\nu)=1,\quad \nu\in N=\{\mu^2\}.
    \end{align*}
  \item (The dual case of case 2)The box does not have an $X$ in it and $\mu^2=\lambda$. Then $N$ is a singleton which is the same as $\mu^1$,
    and the weight is 1:
    \begin{align*}
      w(\mu,\nu)=1,\quad \nu\in N=\{\mu^1\}.
    \end{align*}
  \item The box is empty. $\mu^1=\lambda+\ve_i$ and $\mu^2=\lambda+\ve_j$ for some $i\neq j$. Then $N$ again only contains one element $\mu^1\cup\mu^2$,
    with weight 1:
    \begin{align*}
      w(\mu,\nu)=1,\quad \nu\in N=\{\mu^1\cup\mu^2\}.
    \end{align*}
  \item The box is empty and $\mu^1=\mu^2=\lambda+\ve_i:=\mu$ for some $i$. Then $N$ consists of all possible partitions that are obtained by adding a box to a row no lower than $i$th row of $\mu$. That is
    \begin{align*}
      N=\Lambda^i(\mu).
    \end{align*}
    The weights are
    \begin{align*}
      &w(\mu,\nu)\\
      &=\begin{cases}
	\frac{1-q^{\lambda_{i-1}-\lambda_i-1}}{1-q^{\lambda_{i-1}-\lambda_i}},&\nu=\mu+\ve_i;\\
	\frac{1-q}{1-q^{\lambda_{i-1}-\lambda_i}}q^{\lambda_j-\lambda_i-1}(1-q^{\lambda_{j-1}-\lambda_j}),&\nu=\mu+\ve_j\text{ for some }2\le j< i;\\
	\frac{1-q}{1-q^{\lambda_{i-1}-\lambda_i}}q^{\lambda_1-\lambda_i-1},&\nu=\mu+\ve_1.
      \end{cases}\\&\qquad\qquad\qquad\qquad\qquad\qquad\qquad\qquad\qquad\qquad\qquad\qquad\qquad\qquad,\nu\in N.
    \end{align*}
\end{enumerate}
We call this the \textit{growth graph rule} as opposed to the \textit{insertion rule} described in Sections \ref{s:qrs} and \ref{s:wordqrs}.
Below is a group of illustrations of all 5 cases, where $\lambda$ is the southwest partition.
\begin{center}
  \begin{tikzpicture}[scale=4]
    \def \s{.3}
    \node (100) at (0,0) {
    \begin{tikzpicture}[scale=.3]
      \draw (0,0) -- (4,0) -- (4,-1) -- (2,-1) -- (2,-2) -- (0,-2) -- (0,0);
    \end{tikzpicture}};
    \node (110) at (1,0) {
    \begin{tikzpicture}[scale=.3]
      \draw (0,0) -- (4,0) -- (4,-1) -- (2,-1) -- (2,-2) -- (0,-2) -- (0,0);
    \end{tikzpicture}};
    \node (101) at (0,1) {
    \begin{tikzpicture}[scale=.3]
      \draw (0,0) -- (4,0) -- (4,-1) -- (2,-1) -- (2,-2) -- (0,-2) -- (0,0);
    \end{tikzpicture}};
    \node (1111) at (1+\s,1+\s) {
    \begin{tikzpicture}[scale=.3]
      \draw (0,0) -- (4,0) -- (4,-1) -- (2,-1) -- (2,-2) -- (0,-2) -- (0,0);
      \draw (1,-2) -- (1,-3) -- (0,-3) -- (0,-2);
    \end{tikzpicture}};
    \node (1112) at (1,1) {
    \begin{tikzpicture}[scale=.3]
      \draw (0,0) -- (4,0) -- (4,-1) -- (2,-1) -- (2,-2) -- (0,-2) -- (0,0);
      \draw (3,-1) -- (3,-2) -- (2,-2) -- (2,-1);
    \end{tikzpicture}};
    \node (1113) at (1-\s,1-\s) {
    \begin{tikzpicture}[scale=.3]
      \draw (0,0) -- (4,0) -- (4,-1) -- (2,-1) -- (2,-2) -- (0,-2) -- (0,0);
      \draw (5,0) -- (5,-1) -- (4,-1) -- (4,0) -- (5,0);
    \end{tikzpicture}};
    \tikzstyle{every node}=[sloped,fill=white]
    \draw (100) -- node{$w(\lambda,\mu^1)$} (101) -- node{$1-q^{\lambda_2}$} (1111) -- node{$1-q^{\lambda_2}$} (110) -- node{$w(\lambda,\mu^2)$} (100);
    \draw (101) -- node{$q^{\lambda_2}-q^{\lambda_1}$} (1112) -- node{$q^{\lambda_2}-q^{\lambda_1}$} (110);
    \draw (101) -- node{$q^{\lambda_1}$} (1113) -- node{$q^{\lambda_1}$} (110);
    \node at (.5,.5) {$X$};
    \node at (.5,-.1) {Case 1};
  \end{tikzpicture}
  \begin{tikzpicture}[scale=3.3]
      \node (200) at (0,0) {
      \begin{tikzpicture}[scale=.3]
	\draw (0,0) -- (5,0) -- (5,-1) -- (3,-1) -- (3,-3) -- (2,-3) -- (2,-4) -- (0,-4) -- (0,0);
      \end{tikzpicture}};
      \node (210) at (1,0) {
      \begin{tikzpicture}[scale=.3]
	\draw (0,0) -- (5,0) -- (5,-1) -- (3,-1) -- (3,-3) -- (2,-3) -- (2,-4) -- (0,-4) -- (0,0);
	\draw (3,-3) -- (3,-4) -- (2,-4);
      \end{tikzpicture}};
      \node (201) at (0,1) {
      \begin{tikzpicture}[scale=.3]
	\draw (0,0) -- (5,0) -- (5,-1) -- (3,-1) -- (3,-3) -- (2,-3) -- (2,-4) -- (0,-4) -- (0,0);
      \end{tikzpicture}};
      \node (211) at (1,1) {
      \begin{tikzpicture}[scale=.3]
	\draw (0,0) -- (5,0) -- (5,-1) -- (3,-1) -- (3,-3) -- (2,-3) -- (2,-4) -- (0,-4) -- (0,0);
	\draw (3,-3) -- (3,-4) -- (2,-4);
      \end{tikzpicture}};
      \tikzstyle{every node}=[sloped,fill=white,font=\small]
      \draw (200) -- node{$w(\lambda,\mu^1)$} (201) --  (211) --  (210) -- node{$w(\lambda,\mu^2)$} (200);
      \node at (.5,-.4) {Case 2};
  \end{tikzpicture}
  \begin{tikzpicture}[scale=3.3]
    \node (100) at (0,0) {
    \begin{tikzpicture}[scale=.3]
      \draw (0,0) -- (5,0) -- (5,-1) -- (3,-1) -- (3,-3) -- (2,-3) -- (2,-4) -- (0,-4) -- (0,0);
    \end{tikzpicture}};
    \node (110) at (1,0) {
    \begin{tikzpicture}[scale=.3]
      \draw (0,0) -- (5,0) -- (5,-1) -- (3,-1) -- (3,-3) -- (2,-3) -- (2,-4) -- (0,-4) -- (0,0);
    \end{tikzpicture}};
    \node (101) at (0,1) {
    \begin{tikzpicture}[scale=.3]
      \draw (0,0) -- (5,0) -- (5,-1) -- (3,-1) -- (3,-3) -- (2,-3) -- (2,-4) -- (0,-4) -- (0,0);
      \draw (5,0) -- (6,0) -- (6,-1) -- (5,-1);
    \end{tikzpicture}};
    \node (111) at (1,1) {
    \begin{tikzpicture}[scale=.3]
      \draw (0,0) -- (5,0) -- (5,-1) -- (3,-1) -- (3,-3) -- (2,-3) -- (2,-4) -- (0,-4) -- (0,0);
      \draw (5,0) -- (6,0) -- (6,-1) -- (5,-1);
    \end{tikzpicture}};
    \tikzstyle{every node}=[sloped,fill=white,font=\small]
    \draw (100) -- node{$w(\lambda,\mu^1)$} (101) --  (111) --  (110) -- node{$w(\lambda,\mu^2)$} (100);
    \node at (.5,-.5/3) {Case 3};
  \end{tikzpicture}
  \begin{tikzpicture}[scale=3.3]
      \node (200) at (0,0) {
      \begin{tikzpicture}[scale=.3]
	\draw (0,0) -- (5,0) -- (5,-1) -- (3,-1) -- (3,-3) -- (2,-3) -- (2,-4) -- (0,-4) -- (0,0);
      \end{tikzpicture}};
      \node (210) at (1,0) {
      \begin{tikzpicture}[scale=.3]
	\draw (0,0) -- (5,0) -- (5,-1) -- (3,-1) -- (3,-3) -- (2,-3) -- (2,-4) -- (0,-4) -- (0,0);
	\draw (1,-4) -- (1,-5) -- (0,-5) -- (0,-4);
      \end{tikzpicture}};
      \node (201) at (0,1) {
      \begin{tikzpicture}[scale=.3]
	\draw (0,0) -- (5,0) -- (5,-1) -- (3,-1) -- (3,-3) -- (2,-3) -- (2,-4) -- (0,-4) -- (0,0);
	\draw (4,-1) -- (4,-2) -- (3,-2);
      \end{tikzpicture}};
      \node (211) at (1,1) {
      \begin{tikzpicture}[scale=.3]
	\draw (0,0) -- (5,0) -- (5,-1) -- (3,-1) -- (3,-3) -- (2,-3) -- (2,-4) -- (0,-4) -- (0,0);
	\draw (1,-4) -- (1,-5) -- (0,-5) -- (0,-4);
	\draw (4,-1) -- (4,-2) -- (3,-2);
      \end{tikzpicture}};
      \tikzstyle{every node}=[sloped,fill=white,font=\small]
      \draw (200) -- node{$w(\lambda,\mu^1)$} (201) --  (211) --  (210) -- node{$w(\lambda,\mu^2)$} (200);
      \node at (.5,-.5/3) {Case 4};
    \end{tikzpicture}
  \begin{tikzpicture}[scale=6]
    \def \s{.3}
      \node (200) at (0,0) {
      \begin{tikzpicture}[scale=.3]
	\draw (0,0) -- (5,0) -- (5,-1) -- (3,-1) -- (3,-2) -- (1,-2) -- (1,-3) -- (0,-3) -- (0,0);
      \end{tikzpicture}};
      \node (210) at (1,0) {
      \begin{tikzpicture}[scale=.3]
	\draw (0,0) -- (5,0) -- (5,-1) -- (3,-1) -- (3,-2) -- (1,-2) -- (1,-3) -- (0,-3) -- (0,0);
	\draw (2,-2) -- (2,-3) -- (1,-3) -- (1,-2);
      \end{tikzpicture}};
      \node (201) at (0,1) {
      \begin{tikzpicture}[scale=.3]
	\draw (0,0) -- (5,0) -- (5,-1) -- (3,-1) -- (3,-2) -- (1,-2) -- (1,-3) -- (0,-3) -- (0,0);
	\draw (2,-2) -- (2,-3) -- (1,-3) -- (1,-2);
      \end{tikzpicture}};
      \node (2111) at (1+\s,1+\s) {
      \begin{tikzpicture}[scale=.3]
	\draw (0,0) -- (5,0) -- (5,-1) -- (3,-1) -- (3,-2) -- (1,-2) -- (1,-3) -- (0,-3) -- (0,0);
	\draw (2,-2) -- (2,-3) -- (1,-3) -- (1,-2);
	\draw (3,-2) -- (3,-3) -- (2,-3) -- (2,-2);
      \end{tikzpicture}};
      \node (2112) at (1,1) {
      \begin{tikzpicture}[scale=.3]
	\draw (0,0) -- (5,0) -- (5,-1) -- (3,-1) -- (3,-2) -- (1,-2) -- (1,-3) -- (0,-3) -- (0,0);
	\draw (2,-2) -- (2,-3) -- (1,-3) -- (1,-2);
	\draw (4,-1) -- (4,-2) -- (3,-2) -- (3,-1);
      \end{tikzpicture}};
      \node (2113) at (1-\s+.05,1-\s-.05) {
      \begin{tikzpicture}[scale=.3]
	\draw (0,0) -- (5,0) -- (5,-1) -- (3,-1) -- (3,-2) -- (1,-2) -- (1,-3) -- (0,-3) -- (0,0);
	\draw (2,-2) -- (2,-3) -- (1,-3) -- (1,-2);
	\draw (5,0) -- (6,0) -- (6,-1) -- (5,-1);
      \end{tikzpicture}};
      \node at (.5,-.125) {Case 5};
      \tikzstyle{every node}=[sloped,fill=white,font=\tiny]
      \draw (200) --node{$w(\lambda,\mu^1)$}  (201) -- node{$\frac{1-q^{\lambda_2-\lambda_3-1}}{1-q^{\lambda_2-\lambda_3}}$} (2111) -- node{$\frac{1-q^{\lambda_2-\lambda_3-1}}{1-q^{\lambda_2-\lambda_3}}$} (210) -- node{$w(\lambda,\mu^2)$} (200);
      \draw (201) -- node{$\frac{(1-q)q^{\lambda_2-\lambda_3-1}(1-q^{\lambda_1-\lambda_2})}{1-q^{\lambda_2-\lambda_3}}$} (2112) -- node{$\frac{(1-q)q^{\lambda_2-\lambda_3-1}(1-q^{\lambda_1-\lambda_2})}{1-q^{\lambda_2-\lambda_3}}$} (210) ;
      \draw (201) -- node{$\frac{q^{\lambda_1-\lambda_3-1}-q^{\lambda_1-\lambda_3}}{1-q^{\lambda_2-\lambda_3}}$} (2113) -- node{$\frac{q^{\lambda_1-\lambda_3-1}-q^{\lambda_1-\lambda_3}}{1-q^{\lambda_2-\lambda_3}}$} (210) ;
    \end{tikzpicture}
\end{center}
\begin{proof}[Proof of Theorem \ref{thm:main}]
  It now suffices to show that the above constuction agrees with the definition of the algorithm in the preceeding section. 
  That is, for any connected triplets $(\lambda,\mu^1,\mu^2)$ surrounding the box $(m,k)$ from the southwest, its branching set with corresponding weights according to the insertion rule agree with $N$ and $w(\mu,\nu)$'s according to the growth graph rule.
  To distinguish the context of insertion rule and growth graph rule we also write $\lambda=\lambda^{k-1}$, $\mu^1=\lambda^{k}$, $\mu^2=\tilde\lambda^{k-1}$ and $\nu=\tilde\lambda^k\in N$ in accordance with the definition of the insertion algorithm in Section \ref{s:qrs}.
  We show this by discussing the location of $(m,k)$ in the permutation matrix of $\sigma$, which corresponds to the 5 cases in the growth graph rule.
  \begin{enumerate}
    \item The box $(m,k)$ has an $X$ in it. This means $\sigma_m=k$.
      So we are inserting a $k$ to the tableau at time $m$, hence $\lambda^{k-1}=\tilde\lambda^{k-1}$, i.e. $\lambda=\mu^2$. Moreover since $\sigma$ is a permutation, we have $\sigma_i\neq k$ for all $i<m$, hence $\lambda^{k-1}=\lambda^k$, i.e. $\lambda=\mu^1$ (condition of Case 1 satisfied).  
      Any branch of $(\lambda^{k-1},\lambda^k,\tilde\lambda^{k-1})=(\lambda,\lambda,\lambda)$ is one of the $\lambda+\ve_j$'s such that the weight $w((\lambda,\lambda,\lambda),\lambda+\ve_j)=w_0(k,j)\neq 0$.
      \begin{align*}
	w_0(k,j)=f_0(j;\lambda,\lambda)\prod_{p=j+1}^k(1-f_0(p;\lambda,\lambda))\\=
	\begin{cases}
	  (1-q^{\lambda_{j-1}-\lambda_j})q^{\lambda_j}&\text{ if }j>1\\
	  q^{\lambda_1}&\text{ if }j=1.
	\end{cases}
      \end{align*}
      So the weights agree. All the pruned branches $\lambda+\ve_j$ with $w_0(k,j)=0$ are exactly the $\lambda+\ve_j$'s that are not partitions. Therefore $N$ is exactly the set of all branches of the triplet.
    \item There's no $X$ in $(i,k)$ for $1\le i\le m$. This means $\sigma_i\neq k$ for $i\le m$. 
      Since $\sigma_i\neq k$ for $i<m$, $\lambda^{k-1}=\lambda^k$, i.e. $\lambda=\mu^1$ (condition of Case 2 satisfied).
      Moreover since $\sigma_m\neq k$, the triplet produces only one branch which is equal to $\mu^2$ with weight 1. This agrees with Case 2.
    \item There's no $X$ in $(m,i)$ for $1\le i\le k$. This means $\sigma_m>k$. 
      By the insertion rule, $\lambda^{k-1}=\tilde\lambda^{k-1}$, i.e. $\lambda=\mu^2$ (condition of Case 3 satisfied).
      Again by the same rule the triplet only produces one branch that equals $\mu^1$ with weight 1.
    \item There's one $X$ in each of $(t,k)$ and $(m,s)$ for some $t<m$ and $s<k$. 
      Then on the one hand $\sigma_t=k$ so $\lambda^k=\lambda^{k-1}+\ve_i$ for some $i$, i.e. $\mu^1=\lambda+\ve_i$.
      On the other hand $\sigma_m<k$ so $\tilde\lambda^{k-1}=\lambda^{k-1}+\ve_j$ for some $j$ and $j_{k-1}=j$, i.e. $\mu^2=\lambda+\ve_j$ (Condition of Case 4 or Case 5 satisfied).
      Therefore the branches of the triplet $(\lambda,\lambda+\ve_i,\lambda+\ve_j)$ are in the form of $\tilde\lambda^k=\lambda^k+\ve_{j_k}$ for $j_k\le j$ with weights
      \begin{equation}\label{eq:ij}
	\begin{aligned}
	  &w((\lambda,\lambda+\ve_i,\lambda+\ve_j),\lambda+\ve_i+\ve_{j_k})=w_1(k,j_k)=\\
	  &\begin{cases}
	    f_0(j_k;\lambda,\lambda+\ve_i)\left(\prod_{p=j_k+1}^{j-1}(1-f_0(p;\lambda,\lambda+\ve_i))\right)\\
	    \times(1-f_1(j;\lambda,\lambda+\ve_i))&\text{ if }j_k<j\\
	    f_1(j;\lambda,\lambda+\ve_i)&\text{ if }j_k=j
	  \end{cases}
	\end{aligned}
      \end{equation}
      If $j\neq i$ (condition of Case 4 satisfied), then $f_1(j;\lambda,\lambda+\ve_i)=(1-q^{\lambda_{j-1}-\lambda_j})/(1-q^{\lambda_{j-1}-\lambda_j})=1$.
      Therefore the branch has only one shape equal to $\mu^1+\ve_j$ with weight 1. This agrees with Case 4.
      
      If $j=i$ (condition of Case 5 satisfied), then by \eqref{eq:ij}
      \begin{align*}
	w_1(k,j_k)=
	\begin{cases}
	  \frac{1-q^{\lambda_{i-1}-\lambda_i-1}}{1-q^{\lambda_{i-1}-\lambda_i}};&\text{ if }j_k=i;\\
	  \frac{1-q}{1-q^{\lambda_{i-1}-\lambda_i}}q^{\lambda_j-\lambda_i-1}(1-q^{\lambda_{j-1}-\lambda_j});&\text{ if }2\le j_k< i;\\
	  \frac{1-q}{1-q^{\lambda_{i-1}-\lambda_i}}q^{\lambda_1-\lambda_i-1};&\text{ if }j_k=1.
	\end{cases}
      \end{align*}
      So the weights agree with Case 5. Moreover, all the pruned branches are exactly the $\lambda+\ve_i+\ve_{j_k}$'s that are not partitions. Therefore $N$ is indeed the set of all branches.
  \end{enumerate}
  When inverting a permutation $\sigma$ to $\sigma^{-1}$, the $X$ marks are transposed, so is the weighted graph by the symmetry of the rule, thus we have arrived at the conclusion.
\end{proof}
Note that in both rules although we consider an arbitrary box with index $(m,k)$, the rules do not depend on either $m$ or $k$. This is a key condition for the symmetry property to work and will be generalised in Proposition \ref{p:branch} in the following section.
Below is the growth graph of the permutation $1423$:
\begin{center}
  \begin{tikzpicture}[scale=1]
      \def \l{{0,1,2.6,5.6,11.6}}
      \def \c{.3}

      \node (s1) at (.5,.5) {$X$};
      \node (s2) at (1.8,8.6) {$X$};
      \node (s3) at (4.1,1.8) {$X$};
      \node (s4) at (8.6,4.1) {$X$};

      \tikzstyle{every node}=[font=\small,draw];
      \foreach \i in {0,...,4}
      {
      \node (\i0) at (\l[\i],\l[0]) {$\emptyset$};
      \node (0\i) at (\l[0],\l[\i]) {$\emptyset$};
      }
      \foreach \i in {1,...,4}
      {
      \node (\i1) at (\l[\i],\l[1]) {$1$};
      \node (1\i) at (\l[1],\l[\i]) {$1$};
      }
      \node (22) at (\l[2],\l[2]) {$1$};
      \node (23) at (\l[2],\l[3]) {$1$};

      \node (321) at ({\l[3] - \c},{\l[2] - \c}) {$2$};
      \node (322) at ({\l[3] + \c},{\l[2] + \c}) {$11$};

      \node (331) at ({\l[3] - \c},{\l[3] - \c}) {$2$};
      \node (332) at ({\l[3] + \c},{\l[3] + \c}) {$11$};

      \node (241) at ({\l[2] - \c},{\l[4] - \c}) {$2$};
      \node (242) at ({\l[2] + \c},{\l[4] + \c}) {$11$};

      \node (421) at ({\l[4] - \c},{\l[2] - \c}) {$2$};
      \node (422) at ({\l[4] + \c},{\l[2] + \c}) {$11$};

      \node (341) at ({\l[3] - \c * 3},{\l[4] - \c * 3}) {$3$};
      \node (342) at ({\l[3] - \c},{\l[4] - \c}) {$21$};
      \node (343) at ({\l[3] + \c},{\l[4] + \c}) {$21$};
      \node (344) at ({\l[3] + \c * 3},{\l[4] + \c * 3}) {$21$};

      \node (431) at ({\l[4] - \c * 3.3},{\l[3] - \c * 3.3}) {$3$};
      \node (432) at ({\l[4] - \c},{\l[3] - \c}) {$21$};
      \node (433) at ({\l[4] + \c},{\l[3] + \c}) {$21$};
      \node (434) at ({\l[4] + \c * 3},{\l[3] + \c * 3}) {$111$};

      \node (441) at ({\l[4] - \c * 11},{\l[4] - \c * 11}) {$4$};
      \node (442) at ({\l[4] - \c * 9},{\l[4] - \c * 9}) {$31$};
      \node (443) at ({\l[4] - \c * 7},{\l[4] - \c * 7}) {$31$};
      \node (444) at ({\l[4] - \c * 5},{\l[4] - \c * 5}) {$211$};
      \node (445) at ({\l[4] - \c * 3},{\l[4] - \c * 3}) {$31$};
      \node (446) at ({\l[4] - \c},{\l[4] - \c}) {$31$};
      \node (447) at ({\l[4] + \c},{\l[4] + \c}) {$22$};
      \node (448) at ({\l[4] + \c * 3},{\l[4] + \c * 3}) {$31$};
      \node (449) at ({\l[4] + \c * 5},{\l[4] + \c * 5}) {$211$};
      
      \tikzstyle{every node}=[font=\tiny,fill=white,sloped]

      \foreach \i [count=\j] in {0,...,3}
      {
      \draw (\i0) --  (\j0);
      \draw (0\i) --  (0\j);
      \draw (\i1) --  (\j1);
      \draw (1\i) --  (1\j);
      }

      \foreach \i in {2,...,4}
      {
      \draw (\i0) --  (\i1);
      \draw (0\i) --  (1\i);
      }

      \draw (21) --  (22);
      \draw (22) --  (23);

      \draw (12) --  (22);
      \draw (13) --  (23);

      \draw (23) -- node{$q$} (241);
      \draw (23) -- node{$1-q$} (242);
      \draw (14) -- node{$q$} (241);
      \draw (14) -- node{$1-q$} (242);

      \draw (31) -- node{$q$} (321);
      \draw (31) -- node{$1-q$} (322);
      \draw (22) -- node{$q$} (321);
      \draw (22) -- node{$1-q$} (322);

      \draw (321) --  (331);
      \draw (322) --  (332);
      \draw (23) --  (331);
      \draw (23) --  (332);

      \draw (321) --  (421);
      \draw (322) --  (422);
      \draw (41) --  (421);
      \draw (41) --  (422);

      \draw (331) --  (341);
      \draw (332) --  (342);
      \draw (331) --  (343);
      \draw (332) --  (344);
      \draw (241) --  (341);
      \draw (241) --  (342);
      \draw (242) --  (343);
      \draw (242) --  (344);

      \draw (421) -- node{$q^2$} (431);
      \draw (421) -- node{$1-q^2$} (432);
      \draw (422) -- node{$q$} (433);
      \draw (422) -- node{$1-q$} (434);
      \draw (331) -- node{$q^2$} (431);
      \draw (331) -- node{$1-q^2$} (432);
      \draw (332) -- node{$q$} (433);
      \draw (332) -- node{$1-q$} (434);

      \draw (341) --  (441);
      \draw (341) --  (442);
      \draw (342) --  (443);
      \draw (342) --  (444);
      \draw (343) --  (445);
      \draw (343) -- node{$q/(1+q)$} (446);
      \draw (343) -- node{$1/(1+q)$} (447);
      \draw (344) --  (448);
      \draw (344) --  (449);

      \draw (431) --  (441);
      \draw (432) --  (442);
      \draw (433) --  (443);
      \draw (434) --  (444);
      \draw (431) --  (445);
      \draw (432) -- node{$q/(1+q)$} (446);
      \draw (432) -- node{$1/(1+q)$} (447);
      \draw (433) --  (448);
      \draw (434) --  (449);


  \end{tikzpicture}
\end{center}
\section{More insertion algorithms}\label{s:other}
In \cite{stanley} the growth diagram technique was used to show the symmetry property for the RS algorithm with row insertion.
To row insert a $k$ into a tableau $P$, we again keep $\lambda^0\prec\lambda^1\prec\dots\prec\lambda^{k-1}$ unchanged. Then we append a box at the end of first row of $\lambda^k,\lambda^{k+1},\dots,\lambda^{k_1-1}$, where $k_1$ is the smallest number in row 1 of $P$ that is larger than $k$, then we append a box at the end of second row of $\lambda^{k_1},\lambda^{k_1+1},\dots,\lambda^{k_2-1}$, where $k_2$ is the smallest number in row 2 of $P$ that is larger than $k_1$, and so on and so forth. More precisely, define:
\begin{align*}
  k_0=k;\quad k_j=\min(\{k'> k_{j-1}:\lambda^{k'}_j>\lambda^{k'-1}_j\}\cup\{\ell+1\}),\quad j\ge 1.
\end{align*}
Then the output tableau has:
\begin{align*}
  \tilde\lambda^i=
  \begin{cases}
    \lambda^i,&\text{ if }i<k,\\
    \lambda^i+\ve_j,&\text{ if }k_{j-1}\le i<k_j\text{ for some }j.
  \end{cases}
\end{align*}
For example if we insert a 3 into the tableau \eqref{tab:ex1}, the one-column insertion diagram is as follows:
\begin{center}
  \begin{tikzpicture}
    \tikzstyle{every node}=[fill=white]
    \node (00) at (0,0) {$\emptyset$};
    \node (01) at (0,1) {$2$};
    \node (02) at (0,2) {$2$};
    \node (03) at (0,3) {$31$};
    \node (04) at (0,4) {$41$};
    \node (05) at (0,5) {$42$};
    \node (06) at (0,6) {$421$};
    \node (07) at (0,7) {$422$};
    \node (08) at (0,8) {$4322$};

    \node (10) at (1.1,0) {$\emptyset$};
    \node (11) at (1.1,1) {$2$};
    \node (12) at (1.1,2) {$2$};
    \node (13) at (1.1,3) {$41$};
    \node (14) at (1.1,4) {$42$};
    \node (15) at (1.1,5) {$421$};
    \node (16) at (1.1,6) {$4211$};
    \node (17) at (1.1,7) {$4221$};
    \node (18) at (1.1,8) {$43221$};
    
    \node at (.5,2.55) {$X$};

    \foreach \i in {0,...,8}
    \draw[->] (0\i) -- (1\i);
    \foreach \i [count=\j] in {0,...,7}
    {
    \draw[->] (0\i) -- (0\j);
    \draw[->] (1\i) -- (1\j);
    }

  \end{tikzpicture}
\end{center}
and the corresponding output tableau is
\begin{align*}
  \tilde P=
  \begin{array}{cccc}
    1&1&3&3\\3&4&8&\\5&7&&\\6&8&&\\8&&&
  \end{array}
\end{align*}
The row-insertion of a word is defined in the same way as in column inserting a word.
The normal row and column insertions are related in a few ways.

The most elegant one is the duality. Denote by $(P_\col(w),Q_\col(w))$ and $(P_\row(w),Q_\row(w))$ the tableau pairs when applying column and row insertions to word $w$ respectively.
Also denote by $w^r$ the inverse word of $w$: $w^r=(w_n,w_{n-1},\dots,w_1)$.
Moreover, denote by $\ev(Q)$ the evacuation operation on $Q$, see e.g. \cite{sagan,fulton}. Then
\begin{prop}[see e.g. \cite{fulton}]For any word $w$,
  \begin{align*}
    P_\col(w)=P_\row(w^r).
  \end{align*}
If furthermore $w$ is a permutation, then
\begin{align*}
  Q_\col(w)=(\ev (Q_\row(w^r)))'
\end{align*}
where $T'$ means the transposition of tableau $T$.
\end{prop}
This duality has a matrix input generalisation where $Q_\row(w^r)$ is obtained by a reverse sliding operation, see e.g. \cite{fulton}.

Another simple relation is between the original definitions. 
Row insertion was initially defined as an algorithm of inserting and bumping based on the ordering of integers.
This definition turns into column insertion if one replaces all occurrance of ``row'' by ``column'' and replaces the strong order (greater than) with the weak order (greater than or equal to) due to the asymmetry of the ordering in rows and columns in the definition of a tableau.
For these definitions see e.g. \cite{sagan}.

The third relation is a bit more complicated. In column insertion, we initialise $j_{k-1}=j$. Then in each step we find the largest row index $j_k\le j_{k-1}$ for a letter $k$ such that $\lambda^{k-1}_{j_k-1}>\lambda^k_{j_k}$, append a box to $\lambda^k_{j_k}$ and increase $k$ by 1. 
In row insertion, we initialise $k_0=k$, and in each step one we find the smallest letter $k_j>k_{j-1}$ for a row index $j$ such that $\lambda^{k_j}_j>\lambda^{k_j-1}_j$, append a box to $\lambda^{k_{j-1}}_j,\dots,\lambda^{k_j-1}_j$ and increase $j$ by 1.

In a recent paper \cite{bp13}, a $q$-weighted Robinson-Schensted algorithm with row insertion was proposed. It is defined in a similar way as the $q$-column insertion and has the third relation with the column insertion, which we detail below.

In $q$-column insertion, we initialise $j_{k-1}=k$. Then in each step we run over all row indices $j_k\le j_{k-1}$ for a letter $k$,  with some weight append a box to $\lambda^k_{j_k}$ and increase $k$ by 1.
In $q$-row insertion, we initialise $k_0=k$. In each step one we run over all letters $k_j>k_{j-1}$ for a row index $j$ such that, with some weight we append a box to $\lambda^{k_{j-1}}_j,\dots,\lambda^{k_j-1}_j$ and increase $j$ by 1.

Our definition here is a reformulation equivalent to the one in \cite{bp13}. For $\mu\prec\lambda$ define
\begin{align*}
  g(1;\mu,\lambda)=1-q^{\lambda_1-\mu_1},\\
  g(j;\mu,\lambda)=\frac{1-q^{\lambda_j-\mu_j}}{1-q^{\mu_{j-1}-\mu_j}}\quad j\ge2.
\end{align*}
When $q$-row inserting a $k$ into a tableau $P$, we keep the first $k-1$ shapes unchanged: $\tilde\lambda^i=\lambda^i, i=1,\dots,k-1$, with weight $1$, $\tilde\lambda^k=\lambda+\ve_1$ and for $i>k$, we have a binary branching for each triplet $(\lambda^{i-1},\lambda^i,\tilde\lambda^{i-1}=\lambda^{i-1}+\ve_{j_{i-1}})$:
\begin{align*}
  u((\lambda^{i-1},\lambda^i,\lambda^{i-1}+\ve_{j_{i-1}}),\lambda^i+\ve_{j_i})=
  \begin{cases}
    g(j_{i-1};\lambda^{i-1},\lambda^i),&j_i=j_{i-1}+1;\\
    1-g(j_{i-1};\lambda^{i-1},\lambda^i),&j_i=j_{i-1};\\
    0,&\text{ otherwise.}
  \end{cases}
\end{align*}
Denote by $J_k(P,\tilde P)$ the weight of obtaining a $\tilde P$ after row inserting a $k$ into $P$.

We define the $q$-row insertion for word input in the same way as $q$-column insertion of words: successively $q$-row inserting letters, multiplying the weights, keeping a tableau $Q$ to record changes for each $P$, and merging the same tableau pairs by adding up the weights.
Denote by $\psi_w(P,Q)$ the weight of obtaining tableau pair $(P,Q)$ after $q$-row inserting a word $w$.
Then we have the same recursion rule: for a pair of tableaux $(P,Q)$ with the same shape of size $n$ and a word $w$ of length $n-1$,
\begin{align*}
  \psi_{wk}(P,Q)=\sum\psi_w(\hat P,Q^{n-1})J_k(\hat P,P).
\end{align*}
where the sum is over all tableau $\hat P$ with the same shape as $Q^{n-1}$.
Since the algorithm has a similar triplet branching structure as in Figure \ref{fig:branch}, we can build the one-column insertion construction like in Figure \ref{fig:insertletter}, and concatenate them into a growth graph. With the same approach we can show the symmetry property for this algorithm.
\begin{thm}
  For any permutation $\sigma\in S_n$ and standard tableau pair $(P,Q)$,
  \begin{align*}
    \psi_{\sigma^{-1}}(Q,P)=\psi_\sigma(P,Q).
  \end{align*}
\end{thm}
\begin{proof}[A sketch proof]
  Again denote by $\lambda,\mu^1,\mu^2$ nodes associated with vertices surrounding a box from the south and the west.
The set $M$ of all partitions branched from the triplet with $(u(\mu,\nu):\nu\in M)$ belongs to one of the following 5 cases.
\begin{enumerate}
  \item The box has an $X$ in it. And $\mu^1$, $\mu^2$ are equal to $\lambda$. Then $M$ consists of only one partition $\lambda+\ve_1$ with weight 1.
  \item The box does not have an $X$ in it and $\mu^1=\lambda$. Then $M$ consists of only one partition $\mu^2$ with weight 1.
  \item The box does not have an $X$ in it and $\mu^2=\lambda$. Then $M$ consists of only one partition $\mu^1$ with weight 1.
  \item The box is empty. $\mu^1=\lambda+\ve_i$ and $\mu^2=\lambda+\ve_j$ for some $i\neq j$. Then $M$ again only contains one element $\mu^1\cup\mu^2$ with weight 1.
  \item The box is empty and $\mu^1=\mu^2=\lambda+\ve_i:=\mu$ for some $i$. Then $M=\{\lambda+\ve_i+\ve_{i+1},\lambda+2\ve_i\}$ with weight
    \begin{align*}
      u((\lambda,\lambda+\ve_i,\lambda+\ve_i),\lambda+\ve_{i+1})=
      \begin{cases}
	1-q,&\text{ if }i=1;\\
	\frac{1-q}{1-q^{\lambda_{i-1}-\lambda_i}},&\text{ if }i>1.
      \end{cases}\\
      u((\lambda,\lambda+\ve_i,\lambda+\ve_i),\lambda+2\ve_i)=
      \begin{cases}
	q,&\text{ if }i=1;\\
	1-\frac{1-q}{1-q^{\lambda_{i-1}-\lambda_i}},&\text{ if }i>1.
      \end{cases}
    \end{align*}
\end{enumerate}
The claim is concluded once these symmetric rules are verified to be equivalent to the insertion rule, which is done in the same way as in the proof of Theorem \ref{thm:main}.
\end{proof}

More generally, the symmetry property does not require very strong conditions on the insertion algorithms. Here we give a sufficient condition for an insertion algorithm to have this property.

First we define a \textit{branching insertion algorithm} as an algorithm that has the branching structure as in Figure \ref{fig:branch} when we insert a letter to a tableau. 
That is, there exists an \textit{initial branching weight function} $w_0$, a \textit{high level weight function} $w_1$ and a \textit{low level weight function} $w_2$ such that when inserting a letter $k$, the weight of a new $i$th shape is:
\begin{align*}
  w((\lambda^{i-1},\lambda^i,\tilde\lambda^{i-1}),\tilde\lambda^i)=
  \begin{cases}
    w_0((\lambda^{i-1},\lambda^i,\tilde\lambda^{i-1}),\tilde\lambda^i),&\text{ if }i=k;\\
    w_1((\lambda^{i-1},\lambda^i,\tilde\lambda^{i-1}),\tilde\lambda^i),&\text{ if }i>k;\\
    w_2((\lambda^{i-1},\lambda^i,\tilde\lambda^{i-1}),\tilde\lambda^i),&\text{ if }i<k.
  \end{cases}
\end{align*}

Then we have
\begin{prop}\label{p:branch}
  If a branching insertion algorithm satisfies the following conditions:
  \begin{description}
    \item[(i)] The insertion of $k'$ into a tableau results in increment of one coordinate by 1 in $\lambda^{k'},\lambda^{k'+1},\dots,\lambda^\ell$, while keep $\lambda^0,\lambda^1,\dots,\lambda^{k'-1}$ unchanged, that is, the support of $w_0((\lambda^{i-1},\lambda^i,\tilde\lambda^{i-1}),\tilde\lambda^i)$ and $w_1((\lambda^{i-1},\lambda^i,\tilde\lambda^{i-1}),\tilde\lambda^i)$ are in $\{\lambda^i+\ve_j:j\ge1\}$ and $w_2((\lambda^{i-1},\lambda^i,\tilde\lambda^{i-1}),\tilde\lambda^i)=\mathbb I_{\lambda^i=\tilde\lambda^i}$,
    \item[(ii)] $w_1((\lambda^{m-1},\lambda^m,\lambda^{m-1}+\ve_i),\lambda^m+\ve_j)=\mathbb I_{i=j}$ if $\lambda^m_i=\lambda^{m-1}_i$;
    \item[(iii)] $w_0((\lambda,\lambda,\lambda),\lambda+\ve_i)$ does not depend on the inserted letter;
    \item[(iv)] $w_1((\lambda,\lambda+\ve_i,\lambda+\ve_i),\lambda+\ve_i+\ve_j)$ does not depend on the inserted letter,
  \end{description}
  then it has the symmetry property.
\end{prop}
\begin{proof}[Sketch proof]
  We use the same notations $\lambda,\mu^1,\mu^2$ as in descriptions of growth graph rules such that they surround the box $(m,k)$ from the south and the west.
  The symmetry property is shown once we can construct a symmetric growth graph rule, where the relation of $\lambda$, $\mu^1$ and $\mu^2$ and whether there's an $X$ in the box, determines the location of the box in the permutation matrix, and vice versa. 
  
  Then the 5 cases of the growth diagram rule are satisfied given these conditions: for Case 1, the equivalence between $(\lambda=\mu^1=\mu^2\AND X)$ and $(\sigma_m=k)$ is given by (i) and (ii), and the branched shapes and weights are given by (iii); 
  for Case 2 both the equivalence between $(\lambda=\mu^1\AND\NOT X)$ and $(\sigma_s\neq k\forall s\le m)$ is given by (ii), and the singleton branching with weight $1$ is also given by (ii); 
  for Case 3 both the equivalence between $(\lambda=\mu^2\AND\NOT X)$ and $(\sigma_m>k)$ is given by (i), and the singleton shape with weight $1$ is also given by (i); 
  for Case 4 and 5 the equivalence between $(\mu^1=\lambda+\ve_i\AND\mu^2=\lambda+\ve_j\AND\NOT X)$ and $(\sigma^{-1}_k<m\AND\sigma_m<k)$ is given by (i)(ii), and the singleton shape with weight $1$ in Case 4 is given by (ii) while the branched shapes and weights in Case 5 is given by (iv).
\end{proof}

With this proposition we can test different algorithms for the symmetry property, although not against the symmetry property since it's only a sufficient condition.

For example, one column insertion algorithm proposed by \cite{bp13} satisfies the conditions and has a symmetry property.
See Dynamics 3 in Section 6.5.3 and (8.5) in Section 8.2.1 in that paper.
Here we restate the definition in the framework of branching insertion algorithms.
Denote for $\mu\prec\lambda$:
\begin{align*}
  I^j(\lambda;\mu):=\lambda+\ve_{\max(\{i\le j:\mu_{i-1}>\lambda_i\}\cup\{1\})}.
\end{align*}
The algorithm is defined by the following weights
\begin{align*}
  &w_0((\lambda^{k-1},\lambda^k,\tilde\lambda^{k-1}),\tilde\lambda^k)=\mathbb I_{\tilde\lambda^k=I^k(\lambda^k;\lambda^{k-1})},\\
  &w_1((\lambda^{i-1},\lambda^i,\tilde\lambda^{i-1}),\tilde\lambda^i)\\
  &=\begin{cases}
    -\frac{q^{\lambda^{i-1}_j-\lambda^i_{j+1}+1}(1-q^{\lambda^i_j-\lambda^{i-1}_j})}{(1-q^{\lambda^{i-1}_j-\lambda^i_{j+1}+1})(1-q^{\lambda^{i-1}_{j-1}-\lambda^{i-1}_j})},\\\qquad\qquad\qquad\qquad\text{ if }\tilde\lambda^{i-1}=\lambda^{i-1}+\ve_{j},\tilde\lambda^i=\lambda^i+\ve_{j+1}\text{ for some }j\ge2;\\
    -\frac{q^{\lambda^{i-1}_1-\lambda^i_2+1}(1-q^{\lambda^i_1-\lambda^{i-1}_1})}{1-q^{\lambda^{i-1}_1-\lambda^i_2+1}},\\\qquad\qquad\qquad\qquad\text{ if }\tilde\lambda^{i-1}=\lambda^{i-1}+\ve_1,\tilde\lambda^i=\lambda^i+\ve_2;\\
    1+\frac{q^{\lambda^{i-1}_j-\lambda^i_{j+1}+1}(1-q^{\lambda^i_j-\lambda^{i-1}_j})}{(1-q^{\lambda^{i-1}_j-\lambda^i_{j+1}+1})(1-q^{\lambda^{i-1}_{j-1}-\lambda^{i-1}_j})},\\\qquad\qquad\qquad\qquad\text{ if }\tilde\lambda^{i-1}=\lambda^{i-1}+\ve_{j},\tilde\lambda^i=I^{j}(\lambda^i;\lambda^{i-1})\text{ for some }j\ge2;\\
    1+\frac{q^{\lambda^{i-1}_1-\lambda^i_2+1}(1-q^{\lambda^i_1-\lambda^{i-1}_1})}{1-q^{\lambda^{i-1}_1-\lambda^i_2+1}},\\\qquad\qquad\qquad\qquad\text{ if }\tilde\lambda^{i-1}=\lambda^{i-1}+\ve_1,\tilde\lambda^i=\lambda^i+\ve_1;\\
    0,\\\qquad\qquad\qquad\qquad\text{ otherwise,}
  \end{cases}\\
  &w_2((\lambda^{i-1},\lambda^i,\tilde\lambda^{i-1}),\tilde\lambda^i)=\mathbb I_{\tilde\lambda^i=\lambda^i}.
\end{align*}
We can see from these weight function formulae that this algorithm is different from the $q$-column insertion algorithm in the previous sections. Qualitatively, it has row-insertion like branching of inserting a bumped letter into a lower row, which never happens in normal column insertion where a bumped letter stays in the same row or is inserted to a higher row. As such, it is more natural to compare the $q$-column insertion discussed in the previous sections with the $q$-row insertion, as we have pointed out the similarity of the relation between the $q$-column insertion and the $q$-row insertion and the relation between the normal column and row insertions.

In \cite{bp13} another pair of $q$-RS column and row insertion algorithms called ``$q$-Whittaker-multivariate `dynamics' with deterministic long-range interactions'' were introduced. 
These are also branching insertion algorithms.
We omit the definitions here and point interested readers to Section 8.2.2 of that paper.
These algorithms also satisfy the conditions in Proposition \ref{p:branch} and thus enjoy the symmetry property.

\end{document}